\documentclass[11pt,a4paper]{article}

\usepackage{amsmath, amsfonts, amssymb}

\def\be#1\ee{\begin{equation}#1\end{equation}}

\newtheorem{thm}{Theorem}
\newtheorem{lem}[thm]{Lemma}
\newtheorem{prop}[thm]{Proposition}
\newtheorem{cor}[thm]{Corollary}
\newtheorem{example}[thm]{Example}
\newtheorem{rem}[thm]{Remark}
\newtheorem{defn}[thm]{Defintion}

\DeclareMathOperator{\Var}{Var}

\def\P{{\mathbb{P}}}
\def\R{\mathbb{R}}
\def\E{\mathbb{E}\,}

\def\N{{\mathbb N}}
\def\dd{\mbox{d}}

\newenvironment{proof}[1][] {\noindent {\bf Proof#1:} }{\hspace*{\fill}$\square$\medskip\par}
\newcommand{\ind}{1\hspace{-0.098cm}\mathrm{l}}

\def\DD{{\mathcal D}}
\def\G{{\Gamma}}
\def\tg{{\widetilde{\G}}}
\newcommand{\eps}{\varepsilon}

\def\kk{\kappa_H}

\def\NN{{\mathcal N}}

\def\tl{{\widetilde{\ell}}}

\def\W{{\mathcal W}}
\def\II{{\mathcal{I}}}
\def\LL{{\mathcal{L}}}
\def\MM{{\mathcal{M}}}

\def\functionc{{\mathfrak{C}}}

\def \=L{\ {\buildrel\hbox{\scriptsize d }\over =}\ }

\def\tpsi{{\widetilde{\Psi}}}
\def\tnu{{\widetilde{\nu}}}
\def\tp{{\widetilde{p}}}

\def\FGN{{\textsc{fgn}}}

\title{Small deviations of sums of correlated stationary Gaussian sequences}
\author{Frank Aurzada and Mikhail Lifshits}

\begin{document}

\maketitle


\begin{abstract}
    We consider the small deviation probabilities (SDP) for sums of stationary
    Gaussian sequences. For the cases of constant boundaries and boundaries
    tending to zero, we obtain quite general results. For the case
    of the boundaries tending to infinity, we focus our attention on
    the discrete analogs of the fractional Brownian motion (FBM). It turns out that
    the lower bounds for the SDP can be transferred from the well studied FBM case
    to the discrete time setting under the usual assumptions that imply weak convergence
    while the transfer of the corresponding upper bounds necessarily requires a
    deeper knowledge of the spectral structure of the underlying stationary sequence.
\end{abstract}


\section{Introduction and main results} \label{sec:sec1}
\subsection{Introduction}
The small deviation problem for a stochastic process consists in studying the
probability that the process only has fluctuations below its natural scale.
Small deviation probabilities play a fundamental role in many problems in
probability and analysis, which is why there has been a lot of interest in
small deviation problems in recent years, cf.\ the survey \cite{lishao} and
the literature compilation \cite{lifshitsliterature}. There are many
connections to other questions such as the law of the iterated logarithm of
Chung type, strong limit laws in statistics, metric entropy properties of
linear operators, quantization, and several other approximation quantities
for stochastic processes.

Our work heavily relies on the recently proved Gaussian correlation inequality \cite{royen,latalamatlak}; and we believe that this new tool can lead to the solution of other, formerly inaccessible problems in the area of small deviation probabilites.

In this paper, we study small deviations of sums of correlated stationary
centered Gaussian sequences that are related to Fractional Brownian motion
(FBM). Let us first recall FBM and its small deviation asymptotics.

FBM $(W^H_t)_{t\in \R}$ is a centered Gaussian process with covariance
$$
   \E W^H_t W^H_s = \frac{1}{2} \big(|t|^{2H} + |s|^{2H} - |t-s|^{2H} \big),
$$
where $0<H<1$ is a constant parameter, called Hurst parameter. For $H=1/2$ this
is a usual Brownian motion. For any $0<H<1$, the process has stationary
increments, but no independent increments (unless $H=1/2$). Furthermore, it
is an $H$-self-similar process.  Finally, we recall the small deviation
asymptotics for fractional Brownian motion $W^H$ \cite{lilinde1998}
$$
    \ln \P\{ \sup_{0\le  t\le  1} |W^H(t)|
    \le  \eps \} \sim -\kk \eps^{-1/H},\qquad \text{as $\eps\to 0$,}
$$
where the constant $\kk\in(0,\infty)$ is not known explictly unless $H=1/2$
(and $\kappa_{1/2}=\pi^2/8$). Using the scaling property of FBM, this can
be re-written as
\begin{equation} \label{eqn:fbmconvergence}
     \ln \P\{ \sup_{0\le  t\le  N} |W^H(t)|\le  f_N \}
     \sim -\kk N f_N^{-1/H},\ \text{as $N\to\infty$, $N^{-H} f_N\to 0$.}
\end{equation}

In this paper, we consider the discrete-time analog of fractional Brownian
motion. Let $(\xi_j)_{j\in\N}$ be a real valued stationary centered Gaussian
sequence such that
\begin{equation} \label{eqn:covariances}
   \sum_{j=1}^n \sum_{k=1}^n \E \xi_j \xi_k \sim n^{2H} \ell(1/n),
\end{equation}
with $0<H<1$ and $\ell$ slowly varying at zero. It is well-known (\cite{taqqu})
that (\ref{eqn:covariances}) implies
\begin{equation} \label{eqn:weakconvergence}
     \left(\frac{1}{n^H \ell(1/n)^{1/2}}\sum_{j=1}^{[nt]}
     \xi_j\right)_{t\geq 0} \Rightarrow (W^H_t)_{t\geq 0}
\end{equation}
with fractional Brownian motion $(W^H_t)$. We remark that the same holds
if the $(\xi_j)$ are not necessarily Gaussian, but certain moment
restrictions hold, \cite{whitt}.

The question to be studied is the ``small deviation'' rate of
$S_n:=\sum_{j=1}^n \xi_j$, i.e.
\begin{equation} \label{eqn:sdofsums}
     \P\{ \max_{n=1,\ldots,N} |S_n| \le  f_N \},
     \qquad \text{as $N\to\infty$},
\end{equation}
where  $f_N \ll N^H \ell(1/N)^{1/2}$. As can be seen from the convergence result,
$(S_n)_{1\leq n\leq N}$ has fluctuations of the scale
$N^H \ell(1/N)^{1/2}$, so that indeed we deal with a small deviation
question.

There are three regimes: if $f_N\to\infty$ the small deviation properties
of $(S_n)$ are indeed governed by the same quantities as for FBM (at least
under some regularity assumptions, which are shown to be necessary).
On the other hand, for $f_N\to 0$, we deal with ``very small'' deviations,
and the rate is completely independent of any relation to FBM; we shall prove rather general results here -- in particular, unrelated to (\ref{eqn:covariances}).
In the intermediate case when $f_N$ is constant (or bounded away from zero and
infinity), the behavior is similar to the ``very small'' deviation regime,
and the rate of decay of the small deviation probability is precisely exponential.

Let us mention some related work. The classical case of {\it independent}
$(\xi_j)$ was studied by Chung \cite{chung}, Mogul'ski{\u\i} \cite{mogulskii},
and Pakshirajan \cite{pakshirajan}. In particular, Mogul'skii showed
that if the $(\xi_j)$  are i.i.d.\ centered variables with unit variance,
$f_N \to \infty$ but $N^{-1/2} f_N \to 0$, then, in agreement with
\eqref{eqn:fbmconvergence} for $H=1/2$,
\[
   \ln \P\{\max_{1\le n\le N}|S_n| \le f_N\}
   \sim   -\frac{\pi^2}{8}\ N\, f_N^{-2}.
\]
\medskip

This paper is structured as follows. Sections~\ref{sec:regimeinfty}, \ref{sec:subseconverysmall}, and \ref{sec:regimeconstant}
contain the main results for the three mentioned regimes, respectively. The proofs are given in the subsequent sections.

\subsection{Small deviations related to FBM} \label{sec:regimeinfty}
We first deal with the regime
$$
    f_N\to \infty \quad\text{and}\quad f_N \ll N^H \ell(1/N)^{1/2}
$$
in (\ref{eqn:sdofsums}). Our first main result (Theorem~\ref{thm:mainthm1})
states that a lower bound holds as one would expect from (\ref{eqn:fbmconvergence}).
In order to formulate it, let us recall the definition of the adjoint of a slowly
varying function (see \cite[Section 1.6]{Sen}): for a function $\tl$ that is
slowly varying at infinity, an adjoint function (unique up to asymptotic
equivalence) is a slowly varying function $L(\cdot)$ satisfying the relation
\be \label{adj1}
   L(r)\ \tl(r L(r)) \to 1, \qquad \textrm{as } r\to \infty.
\ee

Now we are ready to state our first main result.

\begin{thm} Let $(\xi_j)_{j\in\N}$ be a real valued stationary centered Gaussian sequence
such that $\eqref{eqn:weakconvergence}$ holds. If $f_N \to\infty$ and $f_N\ll N^H \ell(1/N)^{1/2}$ then
\begin{equation} \label{eqn:lowerbound}
  \liminf_{N\to\infty}  \frac{\ln \P\{\max_{1\le n\le N}|S_n| \le f_N\} }
  { N\, [f_N L(f_N)]^{-1/H}}
  \ge  - \kk,
\end{equation}
where $L(\cdot)$ is a slowly varying function adjoint to the function
$\tl(r):=\sqrt{\ell(r^{-1/H})}$ and $\kappa_H$ is the constant from
$\eqref{eqn:fbmconvergence}$.
\label{thm:mainthm1}
\end{thm}

The proof of this theorem is given in Section~\ref{sec:proofofthm1}.

We shall prove that the corresponding upper bound surprisingly does {\it not} hold
in this generality. In order to obtain the upper bound, one has to assume
more than only the weak convergence to FBM (see Theorem~\ref{thm:mainthm2}
below), as the following negative result shows:

\begin{thm}
 \label{thm:negative} For any $H\in (0,1)$ and any
sequence $f_N$ such that  $f_N\to\infty$ and $N f_N^{-1/H} \to \infty$,
there exists a real valued stationary centered Gaussian sequence $(\xi_j)_{j\in\N}$
such that $\eqref{eqn:covariances}$ holds with $\ell\equiv 1$ but we have
\begin{equation} \label{eqn:strongerupperbound}
   \limsup_{N\to\infty}
   \frac{\ln \P\{\max_{1\le n\le N}|S_n| \le f_N\}}{f_N^{-1/H} N}= 0.
\end{equation}
\end{thm}

The proof of this theorem is given in Section~\ref{sec:proofofthm2}.

In order to obtain the upper bound corresponding to (\ref{eqn:lowerbound}),
-- instead of only assuming weak convergence to FBM -- we make an assumption
about the spectral measure of the sequence $(\xi_j)$.


As above, let $(\xi_j)_{j\in\N}$ be a real valued stationary centered
Gaussian sequence, and denote by $\mu$ the spectral measure:
$$
    \E \xi_j \xi_k = \E \xi_{|j-k|} \xi_0
    = \int_{[-\pi,\pi)} e^{i |j-k| u} \mu(\dd u),\qquad j,k\in\N.
$$
The spectral measure $\mu$ has a (possibly vanishing) component that
is absolutely continuous w.r.t.\ the Lebesgue measure. Let us denote by
$p$ its density, i.e.\ $\mu(\dd u)=: p(u)\dd u + \mu_s(\dd u)$.

Recall that fractional Gaussian noise, defined by
$\xi^\FGN_j := W^H(j) - W^H(j-1)$, is a stationary centered Gaussian
sequence and it has an absolutely continuous spectral measure
with a density $p_\FGN$.  The latter has
a singularity at zero (see e.g.\ \cite{sam}):
\[
   p_\FGN(u) \sim m_H \, |u|^{1-2H}, \qquad u\to 0,
\]
where $m_H=\Gamma(2H+1)\sin(\pi H)/2\pi$.

We assume that the density $p$ of the absolutely continuous component of $\mu$
satisfies
\be \label{fass}
   p(u) \sim m_H \,  \ell(u)\, |u|^{1-2H}, \qquad u\to 0,
\ee
where $\ell(\cdot)$ is a function slowly varying at zero. This means that
the behavior of the density of the absolutely continuous part of the spectral
measure of the sequence $(\xi_j)$ is comparable to the spectral density of
fractional Gaussian noise, up to the slowly varying function $\ell$. It is
well-known (also see \eqref{eqn:unnumbered6} below) that (\ref{fass})
implies (\ref{eqn:covariances}) and thus (\ref{eqn:weakconvergence}).

%

Our second main result can now be formulated as follows.

\begin{thm}  \label{thm:mainthm2}
Let $(\xi_j)_{j\in\N}$ be a real valued stationary centered Gaussain sequence.
Assume that the density of the absolutely continuous component of the spectral
measure satisfies \eqref{fass}. If  $N\to\infty$,
$[L(f_N) f_N]^{-1/H} N \to \infty$, and $f_N\to\infty$, then
\[
   \ln \P\{\max_{1\le n\le N}|S_n| \le f_N\}
   \sim - \kk\, [L(f_N) f_N]^{-1/H} N,
\]
where again $L(\cdot)$ is a slowly varying function adjoint to the function
$\tl(r)=\sqrt{\ell(r^{-1/H})}$  and $\kappa_H$ is the constant from
$\eqref{eqn:fbmconvergence}$.
\end{thm}

The proof of this theorem is given in Section~\ref{sec:proofofthm3}.

\subsection{Very small deviations} \label{sec:subseconverysmall}
As the next step, we look at the opposite regime where
$$
    f_N\to 0.
$$
Let us fix the setup here as follows. As above, we consider a real valued stationary
centered Gaussian sequence $(\xi_j)_{j\in\N}$ with spectral measure $\mu$:
$$
  \E \xi_j \xi_k = \E \xi_{|j-k|} \xi_0
  =    \int_{[-\pi,\pi)} e^{i |j-k| u} \mu(\dd u).
$$
Further, we denote by $p$ the density of the absolutely continuous
component of $\mu$. As before, we will study the sums $S_n:=\sum_{j=1}^n \xi_j$.

It is well-known (see e.g.\ \cite{brockwelldavis}) that the sequence
$(\xi_j)$ is linearly regular if and only if its spectral measure is
absolutely continuous and its density $p$ satisfies the Kolmogorov condition
\be \label{Kolmkrit}
   \int_{-\pi}^\pi \ln p(u) \dd u > - \infty.
\ee
In the following we do not need the notion of regularity directly
but condition \eqref{Kolmkrit} emerges below.


Our main theorem gives the first two terms of the small deviation rate
under assumption \eqref{Kolmkrit}, i.e.\ in the presence of the regular
component. This includes in particular fractional Gaussian noise and
related sequences but does not depend on any precise relation such as
(\ref{fass}).

\begin{thm} \label{thm:verysmalldeviations} Let $(\xi_j)_{j\in\N}$ be
a real valued stationary centered Gaussian sequence with spectral
measure $\mu$ and denote by $p$ the (possibly vanishing) density of
the absolutely continuous component of $\mu$. For $f_N\to 0$ we have:
$$
    \liminf_{N\to\infty}
    \frac{\ln \P\{ \max_{1\le  n\le  N} |S_n|\le  f_N\}}{N \ln f_N^{-1}} \geq -1.
$$
If additionally condition \eqref{Kolmkrit} holds, then
$$
    \ln \P\{ \max_{1\le n\le  N} |S_n|\le  f_N\} = N \ln f_N
    - N\big[ \ln\pi + \frac{1}{4\pi} \int_{-\pi}^\pi \ln p(u) \dd u
    \big] + o(N),
$$
\end{thm}

The proof of this theorem is given in Section~\ref{sec:verysmall}. We remark
that if condition  \eqref{Kolmkrit} does not hold, various different asymptotics may arise.
As an illustration, we mention a few examples with $p=0$. Here, $\approx$
means that the ratio of both quantities is bounded away from zero and infinity.

\begin{example} \label{exa:particularcases}
If $\mu=\delta_0$, then
$$
    \P \lbrace \max_{1\le n \le N} |S_n| \le f_N \rbrace \approx \frac{f_N}{N}.
$$
If $\mu=\delta_{-\pi}$, then
$$
     \P \lbrace \max_{1\le n \le N} |S_n| \le f_N \rbrace \approx f_N.
$$
If $\mu=\delta_{-\pi/2} + \delta_{\pi/2}$, then
$$
\P \lbrace \max_{1\le n \le N} |S_n| \le f_N \rbrace \approx f_N^2.
$$
If $\mu=\delta_0  + \delta_{-\pi} + \delta_{\pi/2}+\delta_{-\pi/2}$,
then
$$
   \P \lbrace \max_{1\le n \le N} |S_n| \le f_N \rbrace
   \approx \frac{f_N^4}{N}.
$$
\end{example}

\subsection{Constant boundary}\label{sec:regimeconstant}

Finally, we look at an intermediate regime including the case where $f_N=f$
is constant. The setup is the same as in Section~\ref{sec:subseconverysmall}:
Consider a real valued stationary centered Gaussian sequence $(\xi_j)_{j\in\N}$
with spectral measure $\mu$, set $S_n:=\sum_{j=1}^n \xi_j$, and denote by $p$
the density of the absolutely continuous component of $\mu$.

\begin{thm} \label{thm:fconst}
Let $(f_N)$ be a positive sequence having a finite positive limit.
Then the following limit exits:
$$
    \lim_{N\to\infty} \frac{1}{N}
    \ln \P\lbrace \max_{1\le n\le N} |S_n|\le  f_N \rbrace \in (-\infty,0].
$$
In particular, for every constant $f>0$ the following limit exists:
$$
    \functionc(f):= \lim_{N\to\infty} \frac{1}{N}
    \ln \P\lbrace \max_{1\le n\le N} |S_n|\le  f\rbrace \in (-\infty,0].
$$

If additionally the Kolmogorov criterion $(\ref{Kolmkrit})$ is satisfied,
then $\functionc(f)<0$.
\end{thm}

We recall that if the Kolmogorov criterion fails, then the rate may well not
be exponential (see Example~\ref{exa:particularcases} above).

The proof of this result is given in Section~\ref{sec:constantboundary}.

\begin{rem} {\rm
This result also sheds a different light on Theorem~\ref{thm:negative}.
Apparently, the counterexamples there depend on two things: on the one
hand, the fact that there is no absolutely continuous component, which
would give more independence, and -- on the
other hand -- the special structure of the singular
component making that part of the process well-approximable.}
\end{rem}

\begin{rem} {\rm In the case when $(\xi_j)$ is a standard normal
i.i.d.\ sequence, $e^{\functionc(f)}$ has a spectral interpretation as
the largest eigenvalue of the self-adjoint linear operator
$R : L_2[-f,f] \mapsto L_2[-f,f]$ given by
\[
   [R g] (x) := \int_{-f}^f \phi(x-y) g(y) \dd y
\]
where $\phi$ is the standard normal density. Namely, let $u(x,n)$ be the
probability to stay in $[-f,f]$ for $n$ steps of a random random walk
with standard Gaussian steps starting at $x$. Formally, $u(x,0) =1$ for
$x\in [-f,f]$. Then
\[
   u(n+1,x)= [R u(n,\cdot)] (x)
\]
and by induction and the spectral theorem one obtains
\[
    u(n,x)= [R^n 1] (x) = \sum_k \lambda_k^n \,\langle \psi_k,1 \rangle\,
    \psi_k(x) \sim  c(x) \lambda_1^n
\]
where $(\lambda_k, \psi_k)$ are pairs of (decreasing) eigenvalues
and eigenfunctions of $R$. We refer to \cite{bormog} for several
variations of this approach that clearly does not seem to work
beyond the case of independent sequences. }
\end{rem}

\section{Proof of Theorem~\ref{thm:mainthm1}} \label{sec:proofofthm1}

\subsection{Preliminaries}

First, we shall make an extensive use of the recently proved Gaussian correlation
inequality \cite{royen,latalamatlak}. It states that for any centered Gaussian measure
$\mu$ on $\R^d$ and any closed, convex, symmetric sets $B_1,B_2$ one has
$$
   \mu(B_1\cap B_2)\geq \mu(B_1)\cdot \mu(B_2).
$$
We shall use it in the form
\begin{eqnarray}  \nonumber
    & & \P\{ \max_{k\in A_1} |X_k|\le  \eps_1;  \max_{k\in A_2} |X_k|\le  \eps_2\}
\\   \label{eqn:gaussiancorrelation}
    & \geq & \P\{ \max_{k\in A_1} |X_k|\le  \eps_1\} \cdot
    \P\{ \max_{k\in A_2} |X_k|\le  \eps_2\},
\end{eqnarray}
for centered Gaussian vectors $X=(X_k)_{1\le k\le d}$ and index sets $A_1,A_2\subseteq \{1,\ldots,d\}$, and $d\in\N$.

Second, we shall recall the extended Talagrand lower bound for
small deviation probabilities, which will be used at various occasions. For this
purpose, let $(X_t)_{t\in T}$  be a centered Gaussian process. We define the Dudley
metric by
$$
   \rho(t,s):=\E[ |X_t-X_s|^2]^{1/2}
$$
and the corresponding covering numbers of $T$ by
$$
     N_c(h):=\min\{ n~|~ \exists t_1, \ldots, t_n \in T :
     \min_{i=1,\ldots,n} \rho(t,t_i) \le  h \textrm{ for all } t\in T  \}.
$$
Then the extended Talagrand bound for small deviations that we shall use
(Theorem~2 from \cite{aurzadalifshits}, see the original Talagrand's version
p.\ 257 in \cite{ledoux} and \cite{talagrand}) says that if for some function
$\Psi$ we have $N_c(h)\le  \Psi(h)$ and
\begin{equation} \label{eqn:talagrandassumption}
    \Psi(h/2)\le  C\, \Psi(h),
\end{equation}
for some $C>1$, then we have
\begin{equation} \label{eqn:talagrandb}
    \log \P\{ \sup_{t,s\in T} |X_t - X_s| \le  c_0 h \} \geq - c \, \widetilde\Psi(h),
\end{equation}
with some numercial constant $c_0$ and constant $c>0$ depending on $C$, and
$$
    \widetilde\Psi(h):=\int_h^{\operatorname{diam} (T,\rho)} \frac{\Psi(u)}{u}\dd u.
$$

Finally, let us introduce the regularly varying function
\begin{equation} \label{eqn:defnd}
  d(r):= r^{1/H} L(r)^{1/H}.
\end{equation}
Then by using \eqref{adj1} we have as $r\to\infty$
\begin{eqnarray}
  d(r)^H \ \sqrt{\ell(d(r)^{-1})}
   &=& r \ L(r)\ \sqrt{\ell(r^{-1/H}L(r)^{-1/H})} \notag
\\
   &=& r \ L(r) \ \tl(r L(r)) \sim r.    \label{drr}
\end{eqnarray}

\subsection{Proof of the theorem}
Let $M>0$ be a large constant, and let $\eps>0$ be a small constant such that
$0<\eps<\tfrac{1}{c_0}$, where $c_0$ is the numerical constant from extended Talagrand
lower bound (\ref{eqn:talagrandb}).

Set $\Delta:=\Delta_N:= \lfloor d( M^H f_N)\rfloor$, where the function $d$
was defined in (\ref{eqn:defnd}), and $A:=\{j\Delta, 0\le j \le \tfrac{N}{\Delta}\}$.

Note that by (\ref{drr}), we have
$\Delta^H \ell(\Delta^{-1})^{1/2} \sim M^H f_N$. Let $N$ be large enough such that
$$
   \Delta^H \ell(\Delta^{-1})^{1/2} \le  (1+\eps) M^H f_N.
$$

Using the Gaussian correlation inequality (\ref{eqn:gaussiancorrelation}), we obtain
\begin{eqnarray*}
   && \P\{\max_{1\le n\le N}|S_n| \le f_N\}
\\
   &\ge&  \P\left\{ \max_{a\in A}|S_a| \le c_0 \eps f_N,
   \max_{a\in A}\ \max_{1\le n \le \Delta}|S_{a+n}-S_a| \le (1-c_0\eps) f_N \right\}
\\
    &\ge&  \P\left\{ \max_{a\in A}|S_a| \le c_0 \eps f_N \right\} \
   \prod_{a\in A}\   \P\left\{\max_{1\le n\le \Delta}|S_{a+n}-S_a|
   \le (1-c_0\eps) f_N \right\}
\\
    &\ge&  \P\left\{ \max_{a\in A}|S_a| \le c_0 \eps f_N \right\} \
   \P\left\{\max_{1\le n\le \Delta}|S_n| \le (1-c_0\eps) f_N \right\}^{N/\Delta+1}
\\
    &=:& \P_1 \ \P_2^{N/\Delta+1}.
\end{eqnarray*}
For $\P_2$, we use weak convergence, with $M$ and $\eps$ fixed and $N$ going to infinity
and obtain:
\begin{eqnarray*}
  \P_2 &=&  \P\left\{\max_{1\le n\le \Delta} \frac{|S_n|}{\Delta^H \ell(\Delta^{-1})^{1/2}}
  \le \frac{(1-c_0\eps) f_N}{\Delta^H \ell(\Delta^{-1})^{1/2}} \right\}
  \\
   &\geq &  \P\left\{\max_{1\le n\le \Delta} \frac{|S_n|}{\Delta^H \ell(\Delta^{-1})^{1/2}}
   \le \frac{(1-c_0\eps)f_N}{(1+\eps)  M^H f_N} \right\}
  \\
   &\geq &  \P\left\{\max_{1\le n\le \Delta} \frac{|S_n|}{\Delta^H \ell(\Delta^{-1})^{1/2}}
   \le \frac{1-c_0\eps}{(1+\eps)  M^H} \right\}
 \\
    &\to&  \P\left\{\max_{0\le t\le 1} |W^H(t)| \le \frac{1-c_0\eps}{(1+\eps) M^H} \right\}.
\end{eqnarray*}
For every fixed $\eps_1>0$ for $M$ large enough by using small deviation asymptotics of FBM
(\ref{eqn:fbmconvergence}) we have
\[
  \P\left\{\max_{0\le t\le 1} |W^H(t)| \le \frac{1-c_0\eps}{(1+\eps)M^H} \right\}
  \ge \exp\left\{ - \kk (1+\eps_1) \frac {M(1+\eps)^{1/H}}{(1-c_0\eps)^{1/H}}  \right\}.
\]

Note that our theorem's assumption 
$\tfrac{f_N}{N^H \ell(1/N)^{1/2}}\to 0$  is equivalent to $N/\Delta\to \infty$. 
Indeed, as we see from \eqref{drr}, the function $d(\cdot)$ is an asymptotic inverse 
(see \cite[Section 1.6]{Sen}) to the function $g:d\mapsto d^H\sqrt{\ell(1/d)}$. Therefore,
\[
   \frac{M f_N}{N^H \ell(1/N)^{1/2}}= \frac{M f_N}{g(N)}\to 0  
   \quad \textrm{is equivalent to}\quad
   \frac{\Delta}{N} \sim  \frac{d(Mf_N)}{d(g(N))} \to 0.  
\]
 

So, since $N/\Delta\sim N M^{-1} [f_N L(f_N)]^{-1/H}$ , it follows that for large $N$
\be \label{P2}
    \P_2^{N/\Delta+1} \ge \exp\left\{ - \kk (1+2\eps_1)(1+\eps)^{1/H}
    \frac {N [f_N L(f_N)]^{-1/H}}{(1-c_0\eps)^{1/H}}   \right\}.
\ee

We continue with the evaluation of $\P_1$ by using the extended Talagrand inequality
(\ref{eqn:talagrandb}) as a tool. Let $N_c(\cdot)$ denote the covering numbers for the
process $\{S_a, a\in A\}$. Weak convergence yields (for large $N$, by using $f_N\to \infty$)
\[
   \E |S_a-S_b|^2 \le C^2 |a-b|^{2H} \ell((a-b)^{-1}), \qquad a,b\in A.
\]
In the following we denote by $C$ large constants, not depending on $N$, that may be
different from line to line. It follows that
\begin{eqnarray*}
   N_c(h) &\le& \Psi(h) := \min \{ |A|, C N h^{-1/H} L(h)^{-1/H}\}
\\
  &\le  & C N\ \min \{ M^{-1} [f_N L(f_N)]^{-1/H},  h^{-1/H} L(h)^{-1/H} \}
\\
   &=&  \begin{cases}
    C N M^{-1} [f_N L(f_N)]^{-1/H},       & h<h_*,\\
    C N  h^{-1/H} L(h)^{-1/H}, &h\ge h_*,
   \end{cases}
\end{eqnarray*}
with $h_* \sim CM^H f_N$. Notice that the main assumption (\ref{eqn:talagrandassumption})
of the extended Talagrand lower bound is verified because
\[
     \frac {\Psi(h/2)}{\Psi(h)} \le  \frac {(h/2)^{-1/H} L(h/2)^{-1/H} }
     {h^{-1/H} L(h)^{-1/H}} \le  C 2^{1/H},\qquad \forall h>0.
\]
Letting
\[
  \tpsi(r):=\int_r^\infty \frac{\Psi(h)}{h} \ \dd h
\]
we have by (\ref{eqn:talagrandb})
$$
  \P_1 = \P\left\{ \max_{a\in A}|S_a| \le c_0 \eps f_N \right\} \ge
  \exp\left\{- C \tpsi(\eps f_N) \right\}.
$$ 
We finally get the key estimate for $\tpsi(\eps f_N)$. Namely,
\begin{eqnarray*}
  \tpsi(\eps f_N) &\le& \left( \int_{\eps f_N}^{h_*} +
  \int_{h_*}^\infty \right) \frac{\Psi(h)}{h} \ \dd h
\\
  &\le  &  C N M^{-1} [f_N L(f_N)]^{-1/H} \ln\left( \frac{h_*}{\eps f_N} \right)
  + C\, N\, h_*^{-1/H} L(h_*)^{-1/H}
\\
  &\le &  C N M^{-1} [f_N L(f_N)]^{-1/H} \ln\left( \frac{C M^H f_N}{\eps f_N} \right)
       + C\, N\, M^{-1} [f_N L(f_N)]^{-1/H}
\\
  &=&  C \left( M^{-1}  \ln\left( \frac{C M^H}{\eps} \right)
       +  M^{-1}\right)  N\, [f_N L(f_N)]^{-1/H}.
\end{eqnarray*}
We conclude that
\be \label{P1}
  \P_1  \ge
  \exp\left\{- C  \left( M^{-1}  \ln\left( \frac{C M^H}{\eps} \right)
   + C\, M^{-1}\right)  N\, [f_N L(f_N)]^{-1/H} \right\}.
\ee
By combining \eqref{P2} and \eqref{P1} we obtain for large $N$ and $M$,
\begin{eqnarray*}
    && \frac{\ln \P\{\max_{1\le n\le N}|S_n| \le f_N\} }{ N\, [f_N L(f_N)]^{-1/H}}
 \\
   &\ge& - \kk (1+2\eps_1)(1+\eps)^{1/H}(1-c_0\eps)^{-1/H}
   - C  \left( M^{-1}  \ln\left( \frac{C M^H}{\eps} \right) +  M^{-1}\right).
\end{eqnarray*}
By letting first $N\to\infty$ and then $M\to\infty$ with $\eps,\eps_1$ fixed, we have
\[
  \liminf_{N\to\infty}   \frac{\ln \P\{\max_{1\le n\le N}|S_n| \le f_N\} }
  { N\, [f_N L(f_N)]^{-1/H}}
  \ge  - \kk (1+2\eps_1)(1+\eps) (1-c_0\eps)^{-1/H}.
\]
Then, letting $\eps,\eps_1\searrow 0$, we obtain
\[
  \liminf_{N\to\infty}  \frac{\ln \P\{\max_{1\le n\le N}|S_n| \le f_N\} }
  { N\, [f_N L(f_N)]^{-1/H}} \ge  - \kk,
\]
as required.

\section{Proof of Theorem~\ref{thm:mainthm2}}
\label{sec:spectrallower}\label{sec:proofofthm3}
\subsection{Preliminaries}
First note that the lower bound in Theorem~\ref{thm:mainthm2} follows from
Theorem~\ref{thm:mainthm1}, since (\ref{fass}) implies (\ref{eqn:covariances})
and thus (\ref{eqn:weakconvergence}). However, we shall still give an independent
proof of the lower bound in Theorem~\ref{thm:mainthm2} for the special case
that $\mu$ is itself absolutely continuous with density $p$. This is because
parts of this proof will be used in the proof of the {\it upper} bound in
Theorem~\ref{thm:mainthm2} and also in the proof of
Theorem~\ref{thm:negative} (namely (\ref{low1a}) and (\ref{low3b}) as well as (\ref{eqn:behofp})).

First, it is useful to evaluate the variances of the partial sums
\begin{eqnarray}
   \E|S_n|^2 &=& \int_{-\pi}^\pi |e^{inu}-1|^2 \frac{p(u)\dd u}{|1-e^{iu}|^2} \notag
\\
   &=& n^{-1} \int_{-n\pi}^{n\pi} |e^{iv}-1|^2 \frac{p(v/n)\dd v}{|1-e^{iv/n}|^2} \notag
\\
   &\sim& \ell(1/n)\, m_H\, n^{-1}
         \int_{-\infty}^{\infty} |e^{iv}-1|^2 (|v|/n)^{-1-2H}  \dd v \notag
\\
   &=&  \ell(1/n)\, n^{2H}.     \label{eqn:unnumbered6}
\end{eqnarray}
In particular, if we let $n\sim d(r)$, we get $\E|S_n|^2\sim r^2$ by \eqref{drr}.

In the following, we split the spectral measure into three pieces by fixing a large
$M>0$ and restricting the spectral measure to the sets $\{|u|< \tfrac{1}{M d(f_N)}\}$,
 $\{\tfrac{1}{M d(f_N)} \le |u|\le \tfrac{M}{d(f_N)}\}$, and
 $\{|u|> \tfrac{M}{d(f_N)}\}$, respectively. The sequence $(\xi_j)$ splits into
 the sum of three independent sequences $\xi^{(1)},\xi^{(2)},\xi^{(3)}$. The corresponding partial sums
 will be denoted $S_{z,n}$ with $z=1,2,3$. From the small deviation viewpoint, $S_{2,n}$
 is the main term, while two others are inessential remainders.

Let us finally mention that at various places below we will use 
the following form of Anderson's inequality. 
Let $(\xi_j)$ and $(\xi_j')$ be stationary real centered Gaussian 
sequences with spectral measures $\mu$ and $\mu'$, respectively, 
and $\mu=\mu'+\nu$ with another (positive) measure $\nu$. If $(S_n)$, 
$(S_n')$  are the partial sums corresponding to $(\xi_j)$, $(\xi_j')$,
then for any $f>0$ it is true that
\be  \label{eqn:anderson}
    \P\{\max_{1\le n\le N}|S_{n}| \le f \} 
    \le 
    \P\{\max_{1\le n\le N}|S_{n}'| \le f \}.
\ee
Indeed, one can construct a probability space with an independent  
stationary real centered Gaussian sequences $(\xi_j'')$ 
and  $(\eta_j)$ such that $(\xi_j'')$  is equidistributed with
$(\xi_j')$ and $(\eta_j)$ has the spectral measure $\nu$. Then the sum  
$(\xi_j''+\eta_j)$ has the spectral measure $\mu'+\nu=\mu$, i.e. it is
equidistributed with $(\xi_j)$. 
If $(S_n'')$ and $(T_n)$ are the partial sums corresponding to $(\xi_j'')$  
and $(\eta_j)$, respectively, we obtain from Anderson's inequality 
\cite[p. 135]{Lif95} that
\begin{eqnarray*}
 \P\{\max_{1\le n\le N}|S_{n}| \le f \} 
 &=&  \P\{\max_{1\le n\le N}|S_{n}'' + T_n| \le f \} 
 \\ 
 &=&  \E_{\eta} [ \P_{\xi''} \{\max_{1\le n\le N}|S_{n}'' + T_n| \le f \}] 
 \\ 
 &\le& \P\{\max_{1\le n\le N}|S_{n}''| \le f \} 
 = \P\{\max_{1\le n\le N}|S_{n}'| \le f \},
\end{eqnarray*}
and \eqref{eqn:anderson} is justified.

 \subsection{Proof of the lower bound}
 We wish to prove that for any fixed $\delta\in(0,\tfrac12)$
 \begin{eqnarray} \label{low1a}
  \ln \P\{\max_{1\le n\le N}|S_{1,n}| \le \delta f_N\}
   &\ge& -  \delta_M [L(f_N) f_N]^{-1/H} N,
 \\ \label{low2a}
  \!\!\!\!\! \ln \P\{\max_{1\le n\le N}|S_{2,n}| \le (1-2\delta) f_N\}
   &\ge& - \kk\, [(1-3\delta) L(f_N) f_N]^{-1/H} N,
 \\ \label{low3b}
  \ln \P\{\max_{1\le n\le N}|S_{3,n}| \le \delta f_N\}
   &\ge& -  \delta_M [L(f_N) f_N]^{-1/H} N,
 \end{eqnarray}
 where $\delta_M$ goes to zero when $M$ tends to infinity.
 It follows from the independence of the $S_{z,.}$, $z=1,2,3$, that
 \begin{eqnarray*}
  && \ln \P\{\max_{1\le n\le N}|S_n| \le f_N\}
\\
   &=& \ln \P\{\max_{1\le n\le N}|S_{1,n}+S_{2,n}+S_{3,n}| \le f_N\}
\\
   &\ge & \ln \P\{\max_{1\le n\le N}|S_{1,n}| \le \delta f_N,
   \max_{0\le n\le N}|S_{2,n}|
   \le (1-2\delta) f_N, \max_{1\le n\le N}|S_{3,n}| \le \delta f_N\}
\\
   &=& \ln \P\{\max_{1\le n\le N}|S_{1,n}| \le \delta f_N\}
       + \ln\P\{ \max_{1\le n\le N}|S_{2,n}| \le (1-2\delta) f_N\}
\\
   && \qquad\qquad +\ln\P\{ \max_{1\le n\le N}|S_{3,n}| \le \delta f_N\}
\\
   &\ge& -  \left[ 2\delta_M + \kk (1-3\delta)^{-1/H} \right]\
       [L(f_N) f_N]^{-1/H} N,
 \end{eqnarray*}
 which provides us with the correct lower bound in Theorem~\ref{thm:mainthm2}.

 \subsubsection{Lower frequencies}
First, we aim at showing (\ref{low1a}). We shall use the extended Talagrand bound
(\ref{eqn:talagrandb}) for the small deviations of $(S_{1,n})_{n\le  N}$.
For this purpose, consider the related Dudley metric: for $n,m\le  N$,
\begin{eqnarray*}
   \E|S_{1,n}-S_{1,m}|^2 &=& \E|S_{1,|n-m|}|^2
\\
   &=&\int_{|u|< \frac{1}{Md(f_N)}}
       \left|\frac{e^{i|n-m|u}-1}{e^{iu}-1} \right|^2 p(u)\dd u
\\
   &\le& C \int_{|u| < \frac{1}{Md(f_N)}}
       \frac{|n-m|^2 u^2}{|e^{iu}-1|^2} |u|^{1-2H} \ell(u)\dd u
\\
   &\le& C |n-m|^2 \int_{|u|\le 1/(Md(f_N))} |u|^{1-2H} \ell(u)\dd u
\\
   &\le& C |n-m|^2 (M d(f_N))^{2H-2} \ell(d(f_N)^{-1}).
\end{eqnarray*}
where $C$ is a constant that does not depend on $N$ (and that may change
from line to line). This shows that
\begin{eqnarray*}
   ( \E|S_{1,n}-S_{1,m}|^2 )^{1/2}&\le & C |n-m| M^{H-1} d(f_N)^{-1}
       d(f_N)^H \ell(d(f_N)^{-1})^{1/2}
\\
   &\le & C |n-m| M^{H-1} d(f_N)^{-1} f_N,
\end{eqnarray*}
where the last step is due to (\ref{drr}). From this bound for the Dudley metric
one obtains for the covering numbers related to the process $S_{1,n}$:
$$
    N_c(\eps) \le C M^{H-1} d(f_N)^{-1} f_N N \eps^{-1},\qquad \eps>0.
$$
This shows, using (\ref{eqn:talagrandb}) that
\begin{eqnarray*}
   \ln \P\{\max_{1\le n\le N}|S_{1,n}| \le \delta f_N\}
   &\geq& - C M^{H-1} d(f_N)^{-1} f_N N \cdot (\delta f_N)^{-1}
\\
   &=& - C \delta^{-1} M^{H-1} d(f_N)^{-1} N
\\
   &=& - C \delta^{-1} M^{H-1}  [L(f_N) f_N]^{-1/H} N,
\end{eqnarray*}
where $\delta_M:=C \delta^{-1} M^{H-1}$ tends to zero as $M\to\infty$,
as required by (\ref{low1a}).

 \subsubsection{Main frequencies}
Using the uniform convergence property of slowly varying functions, cf.\
\cite[Theorem 1.1]{Sen}, in the frequency zone
$\{\tfrac{1}{M d(f_N)} \le |u|\le \tfrac{M}{d(f_N)}\}$ we have
\begin{eqnarray}
  p(u) &\sim&  m_H \,  \ell(u)\, |u|^{1-2H}
  \sim m_H \, \ell(d(f_N)^{-1})\, |u|^{1-2H} \notag
  \\
  &\sim& m_H \, f_N^2\, d(f_N)^{-2H} |u|^{1-2H} \notag
  \\
  &=&  m_H \, L(f_N)^{-2} |u|^{1-2H}. \label{eqn:behofp}
 \end{eqnarray}
 The latter expression coincides with the asymptotics of $p_\FGN$ up to the
 constant factor $L(f_N)^{-2}$.  Therefore, for any $\delta_1>0$ for large $N$
 we have
 \[
   p(u)\le L(f_N)^{-2} (1+\delta_1)^2 p_\FGN(u), \qquad u\in [-\pi,\pi).
 \]
 By Anderson's inquality, cf.\ (\ref{eqn:anderson}), for any $\eps>0$ we have
 \begin{eqnarray*}
   \P\{\max_{1\le n\le N}|S_{2,n}| \le \eps \}
   &\ge&  \P\{\max_{1\le n\le N}| W^H(n)| \le L(f_N) (1+\delta_1)^{-1} \eps \}
 \\
   &\ge& \P\{\sup_{0\le t \le N}| W^H(t)| \le L(f_N) (1+\delta_1)^{-1} \eps \}.
 \end{eqnarray*}
 By letting here $\eps=(1-2\delta) f_N$ and using the small deviation asymptotics
 for $W_H$ from (\ref{eqn:fbmconvergence}), we obtain
 \begin{eqnarray*}
  &&\ln \P\{\max_{1\le n\le N}|S_{2,n}| \le (1-2\delta) f_N\}
  \\
  &\ge& - (1+o(1)) \kk \ [ L(f_N) (1+\delta_1)^{-1} (1-2\delta) f_N ]^{-1/H}  \ N.
  \end{eqnarray*}
By letting $\delta_1\to 0$ this simplifies to
\[
  \ln \P\{\max_{1\le n\le N}|S_{2,n}| \le (1-2\delta) f_N\}
  \ge -  \kk \ [ L(f_N) (1-3\delta) f_N ]^{-1/H}  \ N.
\]
 as announced in \eqref{low2a}.

 \subsubsection{Higher frequencies}
One of the main points of the evaluation here is the uniform bound
for variances. We have, by using \eqref{drr} at the end,
\begin{eqnarray}
    \E|S_{3,n}|^2 &=&  \int_{|u|>\frac{M}{d(f_N)}} |e^{inu}-1|^2
    \frac{p(u)\dd u}{|1-e^{iu}|^2}
    \notag
\\
    &\le&  4 \int_{|u|>\frac{M}{d(f_N)}} \frac{p(u)\dd u}{|1-e^{iu}|^2}
    \notag
\\
    &\sim& 4  \int_{|u|>\frac{M}{d(f_N)}}  \frac{\ell(u)\dd u}{|u|^{1+2H}}
    \notag
\\
   &\le& C \, M^{-2H} \ell\left(\frac{M}{d(f_N)}\right)\,  d(f_N)^{2H}
   \notag
\\
    &\sim& C \, M^{-2H} \, \ell \left( d(f_N)^{-1}\right) \, d(f_N)^{2H}
    \notag
\\
    &\sim&  C  \, M^{-2H} \, f_N^2.  \label{eqn:betterindstar}
\end{eqnarray}

We will show that for every $0\le j< N/d(f_N)$
\be \label{low3a}
     \ln \P\{\max_{j d(f_N) < n \le (j+1)d(f_N)}|S_{3,n}|
     \le \delta f_N\} \ge -\delta_M .
\ee
Once this is done, by the correlation inequality (\ref{eqn:gaussiancorrelation})
and the defintion of $d$, it follows that
\begin{eqnarray*}
    \ln \P\{\max_{1\le n\le N}|S_{3,n}| \le \delta f_N\}
    &\ge& \sum_{j\le \frac{N}{d(f_N)}}
    \ln \P\{\max_{j d(f_N)\le k\le (j+1)d(f_N)}|S_{3,n}| \le \delta f_N\}
 \\
    &\ge&  - \frac {N}{d(f_N)}\ \delta_M  \
    = - \delta_M \, N \ [f_N L(f_N)]^{-1/H},
\end{eqnarray*}
as required in \eqref{low3b}.
Next, for proving \eqref{low3a}, we use the correlation inequality
(\ref{eqn:gaussiancorrelation}) and separate the initial points:
\begin{eqnarray*}
  && \P\{\max_{j d(f_N)\le n\le (j+1)d(f_N)}|S_{3,n}| \le \delta f_N\}
\\
   &\ge& \P\{ |S_{3,j d(f_N)}| \le \frac{\delta f_N}{2} \} \cdot
   \P\{\max_{j d(f_N)\le n\le (j+1)d(f_N)}|S_{3,n}- S_{3,j d(f_N)}|
   \le \frac{\delta f_N}{2}\}
\\
   &=&
   \P\{ |S_{3,j d(f_N)}| \le \frac{\delta f_N}{2} \} \cdot
   \P\{ \max_{1\le n\le d(f_N)}|S_{3,n}| \le \frac{\delta f_N}{2}\}.
\end{eqnarray*}
For the first factor, we use the variance bound (\ref{eqn:betterindstar})
and obtain that (for a standard normal $\NN$)
\[
  \P\{ |S_{3,j d(f_N)}| \le \delta f_N/2 \}
  \ge \P\{ |\NN| \le  C^{-1/2} M^{H} \delta /2 \},
\]
which is close to one for large $M$.

For the second factor let
\[
  \DD:=\int_0^\infty \sqrt{\ln N_c(r)} \, \dd r
\]
be the Dudley integral of covering numbers $N_c(r)$ corresponding
to the process $(S_{3,n})_{1\le n\le d(f_N)}$. Denote
$m$ and $E$ the median and the expectation of
$ \max _{1\le n\le d(f_N)} S_{3,n}$.
It is known from the general Gaussian theory that $m\le E \le C_D \DD$
with $C_D=4\sqrt{2}$, cf. \cite[Section 14, Theorem 1]{Lif95}.
Let also
\[
  \sigma^2:= \max _{1\le n\le d(f_N)} \E|S_{3,n}|^2.
\]
We know from (\ref{eqn:betterindstar}) that
$\sigma^2\le C  \, M^{-2H} \, f_N^2$.
Then for any $r>C_D \DD$ by the Gaussian concentration inequality,
cf.\ \cite[Section 12, Theorem 2]{Lif95},
\begin{eqnarray*}
   \P\{\max _{1\le n\le d(f_N)} |S_{3,n}|\ge r \}
   &\le& 2\, \P\{\max _{1\le n\le d(f_N)} S_{3,n}\ge r \}
\\
   &=&  2\, \P\{\max _{1\le n\le d(f_N)} S_{3,n} -m \ge r-m \}
\\
   &\le&  2\, \P\{\max _{1\le n\le d(f_N)} S_{3,n} -m \ge r-E \}
\\
   &\le&  2\, \P\{\max _{1\le n\le d(f_N)} S_{3,n} -m
        \ge r-C_D\,\DD \}
\\
   &\le&  2\, \P\{\NN \ge (r-C_D\,\DD)/\sigma \}.
\end{eqnarray*}
We apply this bound with $r=\delta f_N/2$.
If we are able to prove that
\be \label{DDbound}
  \DD\le h_M f_N
\ee
with arbitrarily small $h_M$ for large $M$, then we get for large $M$
\begin{eqnarray*}
   \P\{\max _{1\le n\le d(f_N)} |S_{3,n}|\ge \delta f_N/2 \}
   &\le& 2\, \P\{\NN \ge (\delta/2-C_D\, h_M)/\sqrt{C} M^{-H} \}
\\
   &\le& 2\, \P\{\NN \ge \delta M^H / 4\sqrt{C} \}
\\
   &=& \P\{|\NN| \ge \delta M^H / 4\sqrt{C} \},
\end{eqnarray*}
or equivalently
\[
   \P\{\max _{1\le n\le d(f_N)} |S_{3,n}|\le \delta f_N/2 \}
   \ge \P\{|\NN| \le \delta M^H / 4\sqrt{C} \},
\]
which is close to $1$, as required for \eqref{low3a}.

It remains to justify \eqref{DDbound}.
First of all, notice that for all integers $n_1,n_2$ it is true that
\begin{eqnarray*}
   \E|S_{3,n_1}- S_{3,n_2}|^2 &\le& \E|S_{n_1}- S_{n_2}|^2
\\
   &=& \E|S_{|n_1-n_2|}|^2 \sim \ell\left(\frac{1}{|n_1-n_2|}\right)\
   |n_1-n_2|^{2H},
\end{eqnarray*}
having used (\ref{eqn:unnumbered6}). Hence, we have a uniform bound
\[
  \E|S_{3,n_1}- S_{3,n_2}|^2 \le A_1 \ell\left(\frac{1}{|n_1-n_2|}\right)\
  |n_1-n_2|^{2H}.
\]
We will also use the bounds
\[
  \ell\left(\frac{1}{n}\right)\ n^{2H}
  \le A_2 \ell\left(\frac{1}{d}\right)\ d^{2H} \qquad \forall n\le d,
\]
and, by using \eqref{drr}
\[
   \ell\left(\frac{1}{d(r)}\right)\ d(r)^{2H} \le A_3 r^2.
\]
It follows  now that $|n_1-n_2|\le d(r)$ yields
\[
   \E|S_{3,n_1}- S_{3,n_2}|^2 \le A_1 A_2
   \ell\left(\frac{1}{d(r)}\right)\ d(r)^{2H}
   \le A_1A_2A_3 r^2=:(A r)^2.
\]
For the covering numbers of the process $(S_{3,n})_{1\le n \le d(f_N)}$ this means
\[
   N_c(Ar) \le \frac {d(f_N)}{d(r)},\qquad r\geq r_0,
\]
where $r_0$ is the point where the function $d$ starts to be defined.
Further, trivially
\[
   N_c(Ar) \le d(f_N)\qquad r>0.
\]
We shall use the second estimate for $r\in[0,f_N/\ln f_N]$ and the first
estimate for $r>f_N/\ln f_N$. Namely,
\begin{eqnarray*}
   \DD &\le& \int_0^\sigma \sqrt{\ln N_c(\rho)} \dd \rho  
   \\
     &=& A \int_0^{\sigma/A} \sqrt{\ln N_c(Ar)} \dd r
   \\
    &\le& A \int_0^{f_N/\ln f_N} \sqrt{\ln d(f_N)} \,
    \dd r + A \int_{f_N/\ln f_N} ^{\sigma/A}
    \sqrt{\ln \frac {d(f_N)}{d(r)}} \, \dd r \qquad 
   \\
   &\le & C f_N (\ln f_N)^{-1/2} + A f_N \int_{1/\ln f_N}^{\sqrt{C}M^{-H}/A}
        \sqrt{\ln \frac {d(f_N)}{d(f_N v)}} \, \dd v,
\end{eqnarray*}
where we used that $d$ is a regularly varying function (so that $d(f_N)\le  f_N^c$
for large $N$) and that $\sigma\le  \sqrt{C} M^{-H}$ from (\ref{eqn:betterindstar}).
The first term already satisfies the claim (\ref{DDbound}), so we shall look at
the second term now. 

Since the function $d(\cdot)$ is $\tfrac{1}{H}$-regularly varying,
we have, for $N\to \infty$ (and so $f_N v\to \infty$ on the range for $v$
considered here):
\[
  \frac {d(f_N)}{d(f_N v)} = \frac {d(f_N v \cdot v^{-1})}{d(f_N v)} \to  (v^{-1})^{1/H}.
\]
and so the fraction is bounded above by $v^{-2/H}$, say, for large $N$. Hence,
for large $N$
\[
  \DD \le C f_N M^{-H} + A f_N \int_0^{\sqrt{C}M^{-H}/A} \!\!\!\!\!\!\sqrt{2|\ln v|/H}
  \,\dd v\ =: h_M f_N,
\]
as required in \eqref{DDbound}.

\subsection{Proof of the upper bound}

\subsubsection{Fractional Gaussian noise}
First note that for fractional Brownian motion, we can estimate
the continuous time maximum by the discrete time maximum as follows:
Fix $h>0$. Then using the correlation inequality (\ref{eqn:gaussiancorrelation})
\begin{eqnarray*}
   && \P\{ \max_{t\in[0,N]} |W^H_t| \le (1+h) f_N\}
\\
   &\geq& \P\{ \max_{n=1,\ldots,N} |W^H_n| \le f_N,
   \max_{n=1,\ldots,N} \max_{t\in[0,1]} |W^H_{n-1+t}-W^H_{n-1}| \le h f_N\}
\\
   &\geq& \P\{ \max_{n=1,\ldots,N} |W^H_n| \le f_N\} \cdot
   \P\{ \max_{n=1,\ldots,N} \max_{t\in[0,1]} |W^H_{n-1+t}-W^H_{n-1}|
   \le h f_N\}
\\
   &\geq& \P\{ \max_{n=1,\ldots,N} |W^H_n| \le f_N\}
   \cdot\P\{ \max_{t\in[0,1]} |W^H_{t}| \le h f_N\}^N
\\
   &=& \P\{ \max_{n=1,\ldots,N} |W^H_n| \le f_N\} \cdot
   \exp( N \ln (1- \P\{ \max_{t\in[0,1]} |W^H_{t}| > h f_N\}))
\\
   &\geq& \P\{ \max_{n=1,\ldots,N} |W^H_n| \le f_N\}
   \cdot \exp( - 2 N  \P\{ \max_{t\in[0,1]} |W^H_{t}| > h f_N\})
\\
   &\geq & \P\{ \max_{n=1,\ldots,N} |W^H_n| \le f_N\}
   \cdot \exp( - 2 N e^{-h^2 f_N^2/2} )
\\
   &\geq & \P\{ \max_{n=1,\ldots,N} |W^H_n| \le f_N\}
   \cdot \exp( - 3 N h^{-2/H} f_N^{-2/H}),
\end{eqnarray*}
for $N$ large enough. This shows, using the small deviation
asymptotics of FBM (\ref{eqn:fbmconvergence}),
\begin{eqnarray*}
   && \ln \P\{ \max_{n=1,\ldots,N} |W^H_n| \le f_N\}
\\
   &\le& -\kk (1+h)^{-1/H} f_N^{-1/H} N \cdot (1+o(1)) + 3 h^{-2/H} f_N^{-2/H} N
\\
   &\le& -\kk (1+h)^{-1/H} f_N^{-1/H} N \cdot (1+o(1)).
\end{eqnarray*}

Letting $h\to 0$, this proves that
\begin{equation} \label{eqn:fgnupperbound}
   \ln \P\{ \max_{n=1,\ldots,N} |W^H_n| \le f_N\}
   \le -\kk f_N^{-1/H} N \cdot (1+o(1)).
\end{equation}

\subsubsection{Proof of the general upper bound}
The first step is to cut off the part of the spectral measure that
belongs to the singular component. Let $(S_{n}^\mu)$
and $(S_{n}^p)$ be the partial sums of correlated stationary Gaussian
random variables with spectral measures
$\mu(\dd u)=p(u)\dd u + \mu_s(\dd u)$ and $p(u)\dd u$, respectively.
Then, by Anderson's inequality, cf.\ (\ref{eqn:anderson}),
$$
     \P\{ \max_{n=1,\ldots,N} |S_n^\mu|\le f_N\}
     \leq \P\{ \max_{n=1,\ldots,N} |S_n^p|\le f_N\}.
$$
Therefore, we can assume w.l.o.g.\ that $\mu$ is absolutely continuous
and has a spectral density $p$ satisfying (\ref{fass}).

Fix $M>0$ and $\delta>0$. We saw in (\ref{eqn:behofp}) that
on the frequency zone $\{ \frac{1}{M d(f_N)}\le |u| \le \frac{M}{d(f_N)} \}$
we have
$$
   p(u) \sim m_H L(f_N)^{-2} |u|^{1-2H},
$$
which up to the factor $L(f_N)^{-2}$ is the behavior of
fractional Gaussian noise. Therefore, for large $N$, we have
$$
    p(u) \geq (1-\delta)^2 L(f_N)^{-2} p_\FGN(u),
    \qquad \frac{1}{M d(f_N)}\le |u|\le \frac{M}{d(f_N)}.
$$

Let us denote by $S_{2,n}$ the partial sums of correlated
stationary Gaussian  random variables with spectral measure
$p(u)\ind_{\{ \frac{1}{M d(f_N)}\le |u|\le \frac{M}{d(f_N)} \}}$.
Further, let $W^{H,1}$ and $W^{H,2}$ represent the processes
related to the spectral densities
$p_\FGN(u) \ind_{\{\frac{1}{M d(f_N)}\le |u| \le \frac{M}{d(f_N)} \}^c}$
and
$p_\FGN(u) \ind_{\{\frac{1}{M d(f_N)}\le |u| \le \frac{M}{d(f_N)}\} }$,
respectively.

Then by using Anderson's inequality (cf.\ (\ref{eqn:anderson})) twice we get:
\begin{eqnarray*}
    && \P\{ \max_{n=1,\ldots,N} |S_n|\le f_N\}
\\
    &\le& \P\{ \max_{n=1,\ldots,N} |S_{2,n}|\le f_N\}
\\
    &\le& \P\{\max_{n=1,\ldots,N} |(1-\delta)
       L(f_N)^{-1} W^{H,2}_n|\le f_N\}
\\
    &=& \P\{ \max_{n=1,\ldots,N} | W^{H,2}_n|\le (1-\delta)^{-1}
       L(f_N) f_N\}
\\
    &\le& \frac{\P\{ \max_{n=1,\ldots,N} | W^{H,1}_n+W^{H,2}_n|
        \le (1+\delta)(1-\delta)^{-1} L(f_N) f_N\}}
        {\P\{ \max_{n=1,\ldots,N} | W^{H,1}_n|
        \le \delta (1-\delta)^{-1} L(f_N) f_N \}},
\end{eqnarray*}
where we used the correlation inequality (\ref{eqn:gaussiancorrelation})
in the last step.

The first term, by (\ref{eqn:fgnupperbound}) is upper bounded by
\[
    \exp(-\kk [(1+ \delta) (1-\delta)^{-1} L(f_N) f_N ]^{-1/H} N
    \cdot (1+o(1))),
\]
so that if it is true that
\begin{eqnarray} \nonumber
   && \ln \P\{ \max_{n=1,\ldots,N} | W^{H,1}_n|\le \delta (1-\delta)^{-1} L(f_N) f_N \}
\\ \label{eqn:cuttoffhighfrequencypartagain}
   &\geq& - \delta_M [ L(f_N) f_N ]^{-1/H} N (1+o(1)),
\end{eqnarray}
with $\delta_M\to 0$ as $M\to \infty$, we are done with the proof
of the upper bound for $S_n$. However, note that (\ref{low1a}) and
(\ref{low3b}) applied to $W^{H,1}$
imply (\ref{eqn:cuttoffhighfrequencypartagain}).

\section{Proof of Theorem~\ref{thm:negative}} 
\label{sec:nega}
\label{sec:proofofthm2}

\subsection{Preliminaries}

Let us recall the spectral point of view. Recall that FBM is
a process with stationary increments that can be written as
a white noise integral
\[
   W^H(t)=\int_\R (e^{itu}-1) \W(\dd u)
\]
where the control measure of the white noise $\W$ is
$\mu(du)=\tfrac{m_H\,du}{|u|^{2H+1}}$ and
$m_H=\tfrac{\Gamma(2H+1) \sin(\pi H)}{2\pi}$ with $0<H<1$.
The discrete time  fractional Gaussian noise is a stationary
sequence $\xi_j^\FGN:=W^H(j)-W^H(j-1)$. We have
\[
   \xi_j^\FGN = \int_\R e^{iju} (1-e^{-iu}) \W(\dd u), \qquad j\in \N.
\]
Hence the spectral measure of $(\xi_j^\FGN)$ on $[-\pi,\pi)$,
which we denote by $\nu_H$, is the projection of the measure
$\tfrac{m_H\, |1-e^{-iu}|^2\,du}{|u|^{2H+1}}$ by the mapping
\[
   x\mapsto 2\pi \left\{\frac{x}{2\pi}\right\}, \qquad x\in \R,
\]
where $\{\cdot\}$ denotes the fractional part of a real number.

The measure  $\nu_H$ has a density $p_\FGN$ with singularity
\[
   p_\FGN(u) \sim m_H \, |u|^{1-2H}, \qquad u\to 0.
\]

We shall construct a stationary Gaussian sequence with spectral measure
that is a ``pertubation'' (to be defined precisely in the next subsection)
of the spectral measure $\nu_H$ of fractional Gaussian noise. We shall see
that this sequence (in fact any pertubation of $\nu_H$) satisfies
(\ref{eqn:covariances}) and thus (\ref{eqn:weakconvergence}),
see Proposition~\ref{p:invprin}. On the other hand, we show that
(\ref{eqn:strongerupperbound}) can be made true, see (\ref{eqn:requpperb}).

Section~\ref{sec:nega} is structured as follows: in
Subsection~\ref{sec:subpert}, we define what we mean by a pertubation of
$\nu_H$ and show that any pertubation satisfies (\ref{eqn:covariances}).
In Subsection~\ref{sec:subconcr}, we construct a concrete pertubation of
$\nu_H$, while Subsection~\ref{sec:subfin} shows (\ref{eqn:strongerupperbound})
for the sequence arising from that concrete construction.

\subsection{Spectral measure perturbation} \label{sec:subpert}

Along with $\nu_H$ introduce a measure
$\tnu_H(du)=\tfrac{\nu_H(du)}{|e^{iu}-1|^2}$. Clearly, it has a density
$\tp_H(u)=\tfrac{p_H(u)}{|e^{iu}-1|^2}\sim  m_H \, |u|^{-1-2H}$, as
$u\to\infty$. Accordingly,
we have
\[
   \tnu_H[h,\pi] \sim \frac{m_H}{2H}\ h^{-2H},\qquad h\to 0.
\]
We introduce a class of perturbations of $\nu_H$ and show that the same
asymptotics holds for every measure of this class.
\medskip

\begin{defn}  A symmetric measure $\G$ on $[-\pi,\pi]$ is called a
{\it perturbation} of $\nu_H$, if there exists a sequence $u_n\searrow 0$
such that $u_n/u_{n+1}\to 1$ and $\G[u_n,\pi]=\nu_H[u_n,\pi]$.
\end{defn}

This simply means that $\G$ is obtained from $\nu_H$ by redistribution
of the measure within the intervals $[u_{n+1},u_n)$. We stress that $\G$
need not at all be absolutely continuous. On the contrary, a typical
perturbation we will use is a partial discretization of $\nu_H$.

As before, we denote $\tg(du)=\tfrac{\G(du)}{|e^{iu}-1|^2}$.

\begin{lem} \label{l:tail} Let $\G$ be a perturbation of $\nu_H$. Then
\[
   \tg[h,\pi] \sim \frac{m_H}{2H}\ h^{-2H},\qquad h\to 0.
\]
\end{lem}

\begin{proof}
Let
\[
   \theta_n = \frac {\max_{u\in [u_{n+1},u_n]} |e^{iu}-1|}
   {\min_{u\in [u_{n+1},u_n]} |e^{iu}-1|}.
\]
We clearly have $\theta_n\to 1$. For each $n$ we have the bound
\[
  \frac{\tg[u_{n+1},u_n)}{\tnu_H[u_{n+1},u_n)}
  \le  \frac{\G[u_{n+1},u_n)}{\nu_H[u_{n+1},u_n)}\ \theta_n^2
  = \theta_n^2,
\]
and similarly
\[
  \frac{\tg[u_{n+1},u_n)}{\tnu_H[u_{n+1},u_n)}
  \ge  \theta_n^{-2}.
\]
Therefore,
\[
  \frac{\tg[u_{n+1},u_n)}{\tnu_H[u_{n+1},u_n)}
  \to 1,
\]
and so
\[
  \tg[u_n,\pi] \sim \tnu_H[u_n,\pi] \sim \frac{m_H}{2H}\ u_n^{-2H}\ .
\]
Notice also that
\[
  \frac{\tnu_H[u_{n+1},\pi]}{\tnu_H[u_n,\pi]} \to 1.
\]
Finally, for $u\in [u_{n+1},u_n]$ it is true that
\begin{eqnarray*}
   \tg[u,\pi] &\le& \tg[u_{n+1},\pi]
\\
   &=&  \frac{\tg[u_{n+1},\pi]}{\tnu_H[u_{n+1},\pi]} \
        \frac{\tnu_H[u_{n+1},\pi]}{\tnu_H[u_{n},\pi]} \
        \frac{\tnu_H[u_{n},\pi]}{\tnu_H[u,\pi]} \ \tnu_H[u,\pi]
\\
   &\le&  (1+o(1)) \ \frac{m_H}{2H}\ u^{-2H}.
\end{eqnarray*}
The lower bound follows in the same way.
\end{proof}

\begin{cor}
Let $\G$ be a perturbation of $\nu_H$. Then
\be \label{int_u2}
    \int_{[-h,h]} u^2\, \tg(\dd u) \le C \, h^{2-2H},
    \qquad 0< h <\pi.
\ee
\end{cor}

\begin{prop} \label{p:invprin}
Let $\G$ be a perturbation of $\nu_H$. Then for a stationary
sequence with spectral measure $\G$ it is true that
\[
  \E|S_n|^2 \sim n^{2H}, \qquad \textrm{as  } n\to \infty.
\]
\end{prop}

\begin{proof} The spectral representation yields
\begin{eqnarray*}
   \E|S_n|^2 &=& \int_{-\pi}^\pi |e^{inu}-1|^2
   \frac{\G(\dd u)}{|e^{iu}-1|^2}
\\
   &=& \int_{-\pi}^\pi |e^{inu}-1|^2\, \tg(\dd u)
   = 4 \int_{0}^\pi (1-\cos(nu)) \, \tg(\dd u).
\end{eqnarray*}
Let denote $F(u) =\tg[u,\pi]$. Integrating by parts implies
\[
   \int_{0}^\pi (1-\cos(nu))\, \tg(\dd u)
   = n  \int_0^{\pi} \sin(nu) F(u)\, \dd u
   = \int_0^{n\pi} \sin(v) F(v/n)\, \dd v.
\]
Let us fix a large positive {\it odd} integer $V$. Since
$F(\cdot)$ is a decreasing function, for $n\ge V$ we have
\[
   \int_0^{n\pi} \sin(v) F(v/n)\, dv
   \le \int_0^{V\pi} \sin(v) F(v/n) \, \dd v.
\]
Furthermore, for any $\eps>0$ and all large $n$
Lemma \ref{l:tail} yields
\[
   \int_0^{V\pi} \sin(v) F(v/n) \, \dd v
   = \frac{m_H}{2H} \int_0^{V\pi} \sin(v) (v/n)^{-2H}
   (1+\vartheta_n(v)\eps)\, \dd v
\]
with $|\vartheta_n(v)|\le 1$. It follows that
\[
  \int_0^{V\pi} \sin(v) F(v/n)\, \dd v
  \le
  \frac{m_H}{2H} \ n^{2H}  \left[ \int_0^{V\pi} \sin(v)v^{-2H} \dd v
     +\eps \int_0^{V\pi} |\sin(v)|v^{-2H} \dd v\right].
\]
Hence,
\[
  \limsup_{n\to\infty} \frac{\E|S_n|^2}{n^{2H}}
  \le \frac{2m_H}{H} \left[ \int_0^{V\pi} \sin(v)v^{-2H} \dd v
      +\eps \int_0^{V\pi} |\sin(v)|v^{-2H} \dd v \right].
\]
By letting $\eps\to 0$, then $V\to\infty$, we obtain
\[
  \limsup_{n\to\infty} \frac{\E|S_n|^2}{n^{2H}}
  \le \frac{2m_H}{H}  \int_0^{\infty} \sin(v)v^{-2H} \dd v.
\]
The converse estimate
\[
   \liminf_{n\to\infty} \frac{\E|S_n|^2}{n^{2H}}
  \ge \frac{2 m_H}{H}  \int_0^{\infty} \sin(v)v^{-2H} \dd v
\]
is obtained by the same way using large {\it even} $V$'s.
We conclude that
\[
  \E|S_n|^2 \sim \left[ \frac{2 m_H}{H}\
  \int_0^{\infty} \sin v \ v^{-2H} \dd v\right] n^{2H}
  = n^{2H}.
\]
The constant in brackets is equal to $1$ by the definition of
$m_H$ and by the well-known formula for the integral, see
\cite[Formula 858.811]{Dw} (notice by the way that the integral
is {\it absolutely} convergent only for $H>1/2$, otherwise it is
understood as the principal value). This constant also {\it must} be
equal to $1$, since the formula holds also for the non-perturbed case
$\G=\nu_H$ where we have the exact equality $\E|S_n|^2=n^{2H}$.
\end{proof}

\subsection{Construction of the perturbed spectral measure}
\label{sec:subconcr}

We first choose two sequences of positive reals $M_j\nearrow \infty$,
and positive integers $q_j \to\infty$ such that
\be\label{Mq}
  \lim_{j\to\infty} \frac{M_j^2}{q_j} =0.
\ee
Denote $d(r)=r^{1/H}$.
In our evaluations we strongly follow the proof of the lower bound
in our spectral result in Section~\ref{sec:spectrallower} For every
$N$ we may split the spectral domain into three parts
$\{|u|\le \tfrac{1}{M d(f_N)}\}$,
$\{\tfrac{1}{M d(f_N)} \le |u|\le \tfrac{M}{d(f_N)}\}$, and
 $\{|u|\ge \tfrac{M}{d(f_N)}\}$, respectively. The sequence $\xi$
 splits into the sum of three independent sequences $\xi^{(1)},\xi^{(2)},\xi^{(3)}$.
 The corresponding partial sums will be denoted
 $S_{z,n}$ with $z=1,2,3$.

 We know that from the small deviation viewpoint,  $S_{2,n}$ is the main
 term  while two others are inessential remainders. Therefore, the
 main attention should be payed to the central part of the spectrum.

 Let us choose a subsequence $N_j$ increasing to infinity so quickly
 that the corresponding central domains do not overlap, i.e.
 $ \tfrac{1}{M_jd(N_j)} > \tfrac{M_{j+1}}{ d(N_{j+1})}$,
 which can be rewritten as
 \be \label{nonover}
    d(N_{j+1})> M_{j+1} M_j d(N_j).
 \ee
 We also need another growth condition
 \be \label{PN}
   \lim_{j\to\infty}  \frac{|\ln P(M_j,q_j)|}{N_j f_{N_j}^{-1/H}} = 0,
 \ee
 where $P(M,q)$ is a function explicitly defined below in \eqref{PMq}.
 At the moment we do not care about its form but only stress that the
 construction of such $(N_j)$ is possible due to the assumption
 $N f_{N}^{-1/H}\to \infty$.

 Once all sequences are constructed, we build the perturbation $\G$
 of the measure $\nu_H$. For each $j$ we discretize the measure
 $\nu_H$ on the zone
  $\{\tfrac{1}{M_j d(f_{N_j})} \le |u|\le \tfrac{M_j}{d(f_{N_j})}\}$,
  by spreading over it uniformly $q_j$ points
  \[
    t_{j,k}:=  \frac{1}{M_j d(f_{N_j})}  +  \frac{k}{q_j} \
    \left( M_j-\frac{1}{M_j} \right)\ \frac{1}{ d(f_{N_j})} \ ,
    \qquad  0\le k\le q_j,
  \]
  and putting the weights
   $\G\{\pm t_{j,k}\}:=\nu_H[t_{j,k},t_{j,k+1}]$ for $0\le k< q_j$.
  We let $\G=\nu_H$ outside of the zones of perturbation.
  The constructed measure $\G$ is a perturbation of $\nu_H$,
  because by \eqref{Mq}
  \[
    \max_{0\le k< q_j} \left| \frac{t_{j,k+1}}{t_{j,k}}-1 \right|
    \le \frac{M_j(M_j-1/M_j)}{q_j} \to 0, \qquad j\to\infty.
  \]
  Therefore, all evaluations from the previous section apply.
  \medskip

\subsection{Probabilistic evaluations}
\label{sec:subfin}
Now we fix $j$ for a while and eliminate it from the notation.
We will prove that, for the concrete construction in
Section~\ref{sec:subconcr}, we have
 \begin{eqnarray} \label{low1}
    \ln \P\{\max_{1\le n\le N}|S_{1,n}| \le f_N/3\}
    \ge -  \delta_M  f_N^{-1/H} N,
 \\ \label{low2}
     \P\{\max_{1\le n\le N}|S_{2,n}| \le f_N/3\} \ge  P(M,q),
 \\ \label{low3}
     \ln \P\{\max_{1\le n\le N}|S_{3,n}| \le f_N/3\}
     \ge -  \delta_M f_N^{-1/H} N,
 \end{eqnarray}
 where $\delta_M$ goes to zero when $M$ tends to infinity.

 From this, it follows via the correlation inequality
 (\ref{eqn:gaussiancorrelation}) that
 \[
    \ln \P\{\max_{1\le n\le N}|S_n| \le f_N\}
    \ge - 2\delta_M  f_N^{-1/H} N + \ln P(M,q).
 \]
 Now we let $j$ vary. By using $M_j\to\infty$ and \eqref{PN}
 we have the required
 \begin{equation} \label{eqn:requpperb}
  \lim_{j\to\infty} \frac{ \ln \P\{\max_{1\le n\le N_j}|S_n| \le f_{N_j}\}}
  {f_{N_j}^{-1/H} N_j } = 0.
 \end{equation}

 It remains to justify (\ref{low1}), (\ref{low2}), and (\ref{low3}).
 We will not repeat in detail the former evaluations leading to the
 estimates \eqref{low1} and \eqref{low3}, as the same estimates were
 shown in (\ref{low1a}) and (\ref{low3b}). Essentially, it is
 sufficient to check which properties of the spectral measure and
 partial sum variances they use.
 \medskip

 The bound \eqref{low1} for the lower frequencies, based on
 extended Talagrand estimate, uses only \eqref{int_u2} in the form
 \begin{eqnarray*}
   \E|S_{1,n}-S_{1,m}|^2 &=& \int_{|u|
   \le \frac{1}{Md(f_N)}} |e^{i(n-m)u}-1|^2 \tg(\dd u)
\\
   &\le&  (n-m)^2  \int_{|u|\le \frac{1}{Md(f_N)}} u^2\, \tg(\dd u)
\\
   &\le&  C \ (n-m)^2\  [M\, d(f_N)]^{2H-2},
 \end{eqnarray*}
 hence,
 \[
   N_c(\eps) \le C\ M^{H-1} d(f_N)^{H-1}\ \frac{N}{\eps},
 \]
 and (\ref{eqn:talagrandb}) yields
 \begin{eqnarray*}
   \P\{\max_{1\le n\le N}|S_{1,n}| \le f_N/3\}
   &\ge& \exp\left\{- C M^{H-1} d(f_N)^{H-1}\ \frac{N}{f_N} \right\}
   \\
   &=&  \exp \left\{ - C M^{H-1} \ f_N^{-1/H} N \right\},
 \end{eqnarray*}
 as required in \eqref{low1}.
 \medskip

 The bound \eqref{low3} for higher frequencies requires first
 of all the variance evaluation:
 \begin{eqnarray*}
    \E|S_{3,n}|^2 &=& \int_{\frac{M}{d(f_N)}
    \le |u| \le \pi} |e^{inu}-1|^2\,  \tg(\dd u)
 \\
    &\le& 8 \ \tg \left[ \frac{M}{d(f_N)},\pi \right]
 \\
    &\le& C \left( \frac{M}{d(f_N)} \right)^{-2H} = C\ M^{-2H} \ f_N^2,
 \end{eqnarray*}
 where we used Lemma \ref{l:tail} in the last step.

 We also need the increment evaluation
 \[
     \E|S_{3,n_1}- S_{3,n_2}|^2 \le \E|S_{n_1}- S_{n_2}|^2
     = \E|S_{|n_1-n_2|}|^2 \le C |n_1-n_2|^{2H},
 \]
 where we used Proposition \ref{p:invprin} in the last step.
 Then, the estimate \eqref{low3} follows as before, by an
 application of the Dudley integral bound.
 \medskip

 We pass now to \eqref{low2} which is the most delicate part. Let
 \[
   \G_k:= \G \{ t_{k} \} =\nu_H[t_{k}, t_{k+1}], \qquad 0\le k< q.
 \]
 We easily obtain from the definitions of $t_k$ and $\nu_H$ that
 \[
   \G_k \le C\ \left(\frac{1}{M d(f_N)}\right)^{1-2H}\
   \frac{M}{q \ d(f_N)}
   = C \ M^{2H} d(f_N)^{2H-2}\ q^{-1}.
 \]
 Furthermore,
 \begin{eqnarray*}
     \tg_k &:=& \tg\{t_k\}=\frac{\G_k}{|e^{it_k}-1|^2}
 \\
     &\le& \frac{C\ \G_k}{t_k^2} \le C \ M^{2+2H} d(f_N)^{2H} q^{-1}
     = C \ M^{2+2H} f_N^{2}q^{-1}.
 \end{eqnarray*}
 The spectral representation yields
 \[
    S_{2,n}= \sum_{k=0}^{q-1} \sqrt{\tg_k} \
    \left(\xi_k \cos(nt_k)+ \eta_k \sin(n t_k) \right)
 \]
 where $\{\xi_k,\eta_k\}_{0\le k<q}$ is a set of standard
 Gaussian i.i.d.\ random variables. It follows that
 \begin{eqnarray*}
     \sup_{n\in \N} |S_{2,n}| &\le& \max_{k} \sqrt{\tg_k} \
     \sum_{k=0}^{q-1} \left(|\xi_k|+ |\eta_k| \right)
\\
     &\le& C \ M^{1+H} f_N\ q^{-1/2}
     \sum_{k=0}^{q-1} \left(|\xi_k|+ |\eta_k| \right).
 \end{eqnarray*}
 and so
 \begin{eqnarray} \nonumber
   \P\left\{ \sup_{n\in \N} |S_{2,n}| \le f_N/3   \right\}
     &\ge&
    \P\left\{  C \ M^{1+H}  q^{-1/2}
    \sum_{k=0}^{q-1} \left(|\xi_k|+ |\eta_k| \right) \le 1/3\right\}
  \\ \label{PMq}
    &:=& P(M,q),
  \end{eqnarray}
 as required in \eqref{low2}.

\section{Proof of Theorem~\ref{thm:verysmalldeviations}}
\label{sec:verysmall}

{\it Upper bound.} Define the matrix $K\in\mathbb R^{N\times N}$ by
$K_{\ell,m}:=\E\xi_\ell \xi_m$ for $\ell,m=1,\ldots N$ and the function
$\kappa(x_1,x_2,\ldots,x_N):=(x_1,x_1+x_2,\ldots,x_1+\ldots+x_N)$.

If \eqref{Kolmkrit} holds, then $\det K >0$ and we have
$$
   \P\{ \max_{1\le  n\le  N} |S_n|\le  f_N\}
   = \int_{\{ \|\kappa(x)\|_\infty \le  f_N \}}
   \frac{1}{(2\pi)^{N/2} \sqrt{\det K}} e^{-\langle x,K^{-1} x\rangle/2}\,
   \dd x.
$$

Since $K$ is non-negative definite, it is true that
$\langle x,K^{-1} x\rangle \geq 0$ for all $x\in \R^N$, so that we get
the following upper bound
$$
    \P\{ \max_{1\le  n\le  N} |S_n| \le  f_N\}
    \le   \frac{\operatorname{vol}\{ ||\kappa(x)||_\infty \le  f_N \}}{(2\pi)^{N/2}
    \sqrt{\operatorname{det} K}} = \frac{( 2f_N)^N }{(2\pi)^{N/2}
    \sqrt{\operatorname{det} K}}\, .
$$
By the Szeg\H{o} limit theorem \cite{boettchersilbermann,gray},
$$
   \lim_{N\to\infty} \frac{1}{N}\ln \operatorname{det} K
   = \frac{1}{2\pi} \int_{-\pi}^\pi \ln [2\pi \,p(u)] \dd u,
$$
where $p$ is the density of the absolutely continuous part of the
spectral measure.  If the integral on the right hand side is finite,
we get that
$$
    \ln \P\{ \max_{1\le  n\le  N} |S_n| \le  f_N\}
    \le  N \ln f_N - N  \big[  \ln\pi
    + \frac{1}{4\pi} \int_{-\pi}^\pi \ln p(u) \dd u\big] + o(N).
$$

{\it Lower bound.} First observe that
\begin{equation} \label{eqn:xinftytozero}
    \{x : \|\kappa(x)\|_\infty \le  f_N \}
    \subseteq \{ x : \|x\|_\infty \le  2 f_N\}.
\end{equation}

Write $(\lambda_j)_{j=1,\ldots,N}=(\lambda_j^{(N)})_{j=1,\ldots,N}$ for
the eigenvalues of $K=K^{(N)}$ and note that a covariance matrix $K$ is
diagonizable (since it is symmetric), say $K=Q^T D Q$ with orthonormal
matrix $Q$ and diagonal matrix $D$. Therefore, for
$x\in \R^N$ it is true that
\begin{eqnarray*}
    \langle x,K^{-1} x\rangle &=& \langle x,Q^T D^{-1} Q x\rangle
    =\langle Q x, D^{-1} Q x\rangle = \sum_{j=1}^N \lambda_j^{-1} (Q x)_j^2.
\\
    &\le  &  \max_{j=1,\ldots,N} \lambda_j^{-1} \cdot \sum_{j=1}^N  (Q x)_j^2
\\
    &=& \frac{1}{\min_{j=1,\ldots,N} \lambda_j} \cdot  \|Qx\|_2^2.
\\
    &=& \frac{1}{\min_{j=1,\ldots,N} \lambda_j} \cdot \|x\|_2^2
\\
    &\le & \frac{1}{\min_{j=1,\ldots,N} \lambda_j} \cdot N  \|x\|_\infty^2\, .
\end{eqnarray*}
In particular, for $x$ from the sets in (\ref{eqn:xinftytozero}),
\[
     \langle x,K^{-1} x\rangle
     \le  \frac{1}{\min_{j=1,\ldots,N} \lambda_j}  \cdot N \cdot (2 f_N)^2.
\]

In order to estimate the minimum of the eigenvalues, recall from linear
algebra that
$$
   \min_j \lambda_j = \min_{x\in\R^N}
   \frac{\langle x,K x\rangle}{\langle x,x\rangle}.
$$

Now, note that by the spectral representation
\[
    {\langle x,K x\rangle} =  \E \left| \sum_{k=1}^N x_k \xi_k \right|^2
      =  \int_{[-\pi,\pi)} \Big|\sum_{k=1}^N x_k e^{iku} \Big|^2 \mu(\dd u)
\]
and by the same formula with $K$ replaced by the unit matrix,
$$
   {\langle x,x\rangle} = \frac{1}{2\pi}\, \int_{[-\pi,\pi)}
   \big|\sum_{k=1}^N x_k e^{i ku} \big|^2 \dd u.
$$

Fix $\delta>0$. Let us denote by $(\widetilde S_n)$ the sequence of partial
sums corresponding to the spectral measure $\widetilde\mu := \mu+\delta \Lambda$
where $\Lambda$ is the Lebesgue measure. Let $\widetilde{K}$ denote the
corresponding covariance matrix and $(\widetilde{\lambda}_j)_{j=1,\dots,N}$
its eigenvalues.

Using the last three observations for $\widetilde K$ and
$(\widetilde \lambda_j)$, we get that
\[
   \min_j \widetilde\lambda_j = \min_{x\in\R^N} \frac
   {\int_{[-\pi,\pi)} \big|\sum_{k=1}^N x_k e^{i\,ku} \big|^2 \widetilde\mu(\dd u)}
   {  \frac{1}{2\pi} \int_{[-\pi,\pi)} \big|\sum_{k=1}^N x_k e^{i\,ku} \big|^2 \dd u}
   \geq 2\pi\delta.
\]
Now we can proceed similarly to the upper bound:
\begin{eqnarray}
     \P\{ \max_{1\le n\le N} |\widetilde S_n|\le  f_N\}
     &=& \int_{\{\|\kappa(x)\|_\infty \le  f_N \}}
     \frac{1}{(2\pi)^{N/2} \sqrt{\det \widetilde K}} \,
     e^{-\langle x,\widetilde K^{-1} x\rangle/2}\, \dd x
\notag
\\
    &\geq& \int_{\{ \|\kappa(x)\|_\infty \le  f_N \}}
    \frac{1}{(2\pi)^{N/2} \sqrt{\det \widetilde K}}\,
    e^{-N f_N^2 / (\pi\delta)}\, \dd x
\notag
\\
 \label{eqn:lowerboundverysmallargument}
    &=& \frac{(2f_N)^N}{(2\pi)^{N/2}
    \sqrt{\det \widetilde K}}\, e^{-N f_N^2 / (\pi\delta)}.
\end{eqnarray}
By using Anderson's inequality, cf.\ (\ref{eqn:anderson}), and applying again the Szeg\H{o}
limit theorem for handling $\det \widetilde K$ we see that
\begin{eqnarray}
&&
   \ln\P\{ \max_{1\le  n\le  N} |S_n|\le  f_N\}
   \geq \ln \P\{ \max_{1\le  n\le  N} |\widetilde S_n|\le  f_N\}
   \notag
\\
   &\geq&  N \ln f_N - N  \Big[\ln\pi +
   \frac{1}{4\pi} \int_{-\pi}^\pi \ln (p(u)+\delta) \dd u
    + o(1)- \frac{f_N^2}{\pi\delta} \Big].
    \label{eqn:lowr}
\end{eqnarray}
By using $f_N\to 0$ we have $f_N^2/\delta = o(1)$ and the first claim
of the theorem,
$$
    \liminf_{N\to\infty} \frac{\ln \P\{ \max_{1\le n\le  N} |S_n|\le  f_N\}}
    {N \ln f_N^{-1}} \geq -1,
$$
follows from \eqref{eqn:lowr}.

Inequality (\ref{eqn:lowr}) also yields
$$
         \liminf_{N\to\infty}
         \frac{\ln \P\{ \max_{1\le  n\le  N} |S_n|\le  f_N\} - N \ln f_N}{N}
         \geq -\ln\pi - \frac{1}{4\pi}
         \int_{-\pi}^\pi \ln (p(u)+\delta) \dd u.
$$
In the case $\int_{-\pi}^{\pi} \ln p(u) \dd u > -\infty$,
letting here $\delta\to 0$ shows the second claim of the theorem.


\section{Constant boundary} \label{sec:constantboundary}

Let $M_N=\max_{1\le n\le N}|S_n|$. We want to prove in particular
that for any constant $f>0$ there exists the limit
\begin{equation} \label{e1}
   \lim_{N\to\infty}
   \frac{\ln \P\{M_N \le f\}}{N} \in (-\infty, 0].
\end{equation}

First we prove an intermediate result (Proposition~\ref{p:incr}) showing that
the limit exists for a certain class of strictly increasing functions.  Then
we show that the limit exists for constants
(Proposition~\ref{prop:existenceoflimitconst}) before finally proving the
full main result (Theorem~\ref{thm:fconst}).

\begin{prop} \label{p:incr}
Let $(f_N)$ be a strictly increasing, positive sequence satisfying the growth
condition
\be \label{incr}
  \sum_{m=1}^{\infty} r_m<\infty
\ee
where
\[
  r_m := 2^{-m}   \max_{2^m\le b \le 2^{m+1}} |\ln (f_b-f_{7b/8})|.
\]

Then there exists the limit
\begin{equation} \label{e2}
   \lim_{N\to\infty}
   \frac{\ln \P\{M_{N} \le f_N \}}{N} \in (-\infty, 0].
\end{equation}
\end{prop}

\begin{rem} \label{rem14} {\rm
Condition \eqref{incr} holds, for example, for the boundaries of the type
$f_N=f-c\, N^{-q}$, $N>N_0$ for any parameters $f,c,q>0$.
}
\end{rem}

\begin{proof}
Let us set
\[
   L_N:=  \frac{\ln \P\{M_{N} \le f_N \}}{N}, \qquad N\in \N.
\]

Let $\Delta_m:=\{N\in\N: 2^m\le N\le 2^{m+1}\}$ be the binary blocks
and denote
\[
 \II_m:=\min_{N\in \Delta_m} L_N, \qquad
 \MM_m:=\max_{N\in \Delta_m} L_N.
\]
It is sufficient for us to prove that there exist equal limits
\be \label{limits}
  \lim_{m\to\infty} \MM_m = \lim_{m\to\infty} \II_m \in(-\infty,0].
\ee
We first prove that the second limit exists.

For any $N_1,N_2\in \N$, $f,\delta>0$, the correlation inequality yields
\be \label{ci}
  \P\{M_{N_1+N_2} \le f+\delta\} \ge \P\{M_{N_1} \le f+\delta\}\
  \P\{|S_{N_1}| \le \delta\}\ \P\{M_{N_2} \le f\}.
\ee
This will be our main tool along the proof.

For any $m\in \N$ let us take $b\in \Delta_{m+1}$ such that
$L_b=\II_{m+1}$ and represent it in the form $b=b_1+b_2$ with
$b_1=b_2=b/2$ for even $b$ and $b_1=(b+1)/2, b_2=(b-1)/2$ for odd $b$.
In any case we have $b_1,b_2\in \Delta_m$. By applying \eqref{ci}
with $N_1=b_1, N_2=b_2$
and $f=f_{b_2}$, $\delta=f_b-f_{b_2}$ we obtain
\begin{eqnarray*}
\P\{M_{b} \le f_b\} &\ge& \P\{M_{b_1} \le f_{b}\}\
  \P\{|S_{b_1}| \le f_b-f_{b_2}\}\ \P\{M_{b_2} \le f_{b_2}\}
\\
   &\ge& \P\{M_{b_1} \le f_{b_1}\}\
  \P\{c b_1|\NN| \le f_b-f_{7b/8}\}\ \P\{M_{b_2} \le f_{b_2}\}
\\
   &\ge& \P\{M_{b_1} \le f_{b_1}\}\
  c \min\{\frac{f_b-f_{7b/8}}{b_1}, 1\}  \ \P\{M_{b_2} \le f_{b_2}\},
\end{eqnarray*}
where $\NN$ is a standard normal random variable.
By taking logarithms, this leads to
\begin{eqnarray*}
  b L_b &\ge& b_1 L_{b_1} + b_2 L_{b_2} - c
              - \left|\ln \frac{f_b-f_{7b/8}}{b_1}\right|
\\
   &\ge& b_1 L_{b_1} + b_2 L_{b_2} - c - \left|\ln(f_b-f_{7b/8})\right| -\ln b_1
\\
\end{eqnarray*}
and we obtain
\begin{eqnarray} \nonumber
\II_{m+1} &=& L_b
\ge \frac{b_1}{b} L_{b_1} + \frac{b_2}{b} L_{b_2}
   - \frac{1}{b}\, \left[ c + \left|\ln(f_b-f_{7b/8})\right| +\ln b_1\right]
\\ \label{e4}
  &\ge& \II_m  - c \, r'_m
\end{eqnarray}
where $r'_m=r_m+ m 2^{-m}$.
It follows from \eqref{e4} that for any $m_0\in \N$
\[
   \liminf_{m\to\infty} \II_m \ge \II_{m_0} - c\, \sum_{m=m_0}^\infty r'_m
\]
where the series is convergent by \eqref{incr}. We observe from this inequality
that $\liminf_{m\to\infty} \II_m$ is not equal to $-\infty$. Furthermore, taking
limsup in the right hand side we obtain
\[
  \liminf_{m\to\infty} \II_m \ge \limsup_{m_0\to\infty} \II_{m_0}\ .
\]
Therefore, the existence of $\lim_{m\to\infty} \II_m$ and its finiteness are
now proved. It remains to prove that
\be \label{limits2}
  \limsup_{m\to\infty} \MM_m \le \lim_{m\to\infty} \II_m.
\ee

For any $m\in \N$ choose $a\in \Delta_{m-1}$, $b\in \Delta_{m+1}$
such that $L_a=\MM_{m-1}$ and $L_b=\II_{m+1}$.
Notice that
\[
   2^{m+2}\ge b\ge b-a \ge 2^{m+1}-2^{m}=2^m,
\]
hence, $b-a \in \Delta_{m}\cup \Delta_{m+1}$. We also have
\[
  \frac{b-a}{b}= 1- \frac{a}{b}\le \frac {7}{8}.
\]

By applying \eqref{ci} with $N_1=a$, $N_2=b-a$, $f=f_{b-a}$, $\delta=f_b-f_{b-a}$
we obtain
\begin{eqnarray*}
  \P\{M_{b} \le f_b\} &\ge& \P\{M_{a} \le f_b\}\
  \P\{|S_{a}| \le f_b-f_{b-a}\}\ \P\{M_{b-a} \le f_{b-a}\}
\\
&\ge& \P\{M_{a} \le f_a\}\
  \P\{|S_{a}| \le f_b-f_{b-a}\}\ \P\{M_{b-a} \le f_{b-a}\}.
\end{eqnarray*}
By taking logarithms we get
\begin{eqnarray*}
   b \II_{m+1} &=& b L_b \ge a L_a + (b-a)L_{b-a} +\ln \P(c\,a|\NN|\le f_b-f_{b-a})
\\
   &\ge& a \MM_{m-1} + (b-a)\min\{\II_m,\II_{m+1}\} +\ln \P(c\,a|\NN|\le f_b-f_{7b/8})
\\
   &\ge& a \MM_{m-1} + (b-a)\min\{\II_m,\II_{m+1}\} -c- \left|\ln [\frac{f_b-f_{7b/8}}{a}]\right|.
\end{eqnarray*}
We may rewrite this inequality as
\begin{eqnarray*}
 a \MM_{m-1} &\le&  b \II_{m+1} - (b-a)\min\{\II_m,\II_{m+1}\}
 + c + \left|\ln [\frac{f_b-f_{7b/8}}{a}]\right|
 \\
  &=&  a \II_{m+1} + (b-a)[ \II_{m+1}
   - \min\{\II_m,\II_{m+1}\}] + c + \left|\ln [\frac{f_b-f_{7b/8}}{a}]\right|
\\
  &\le&  a \II_{m+1} + (b-a)|\II_{m+1} - \II_m| + c + \left|\ln [\frac{f_b-f_{7b/8}}{a}]\right|.
\end{eqnarray*}
Dividing by $a$ and using $(b-a)/a\leq 8$ yields
\begin{eqnarray*}
   \MM_{m-1} &\le&  \II_{m+1} + \frac{b-a}{a}\ |\II_{m+1} - \II_m|
   + \frac{1}{a} \left( c + \left|\ln(f_b-f_{7b/8})\right|+\ln a\right)
  \\
   &\le&   \II_{m+1} + 8\ |\II_{m+1} - \II_m| + c \, r'_{m+1}.
\end{eqnarray*}
Since $\lim_{m\to\infty} (\II_{m+1} - \II_m)=0$ and
$\lim_{m\to\infty} r'_m=0$,
taking the limit yields \eqref{limits2}.
\end{proof}

Now we may prove the main result of this section.

\begin{prop} \label{prop:existenceoflimitconst}
For every constant $f>0$ the following limit exists:
\begin{equation} \label{flim}
   \lim_{N\to\infty}
   \frac{\ln \P\{M_{N} \le f \}}{N} \in (-\infty, 0].
\end{equation}
\end{prop}

\begin{proof}
We shall make use of the log-concavity of Gaussian measures. For any
sequence $(f_N)$ let
\[
  \LL((f_N)):=  \lim_{N\to\infty} \frac{\ln \P\{M_{N} \le f_N \}}{N}
\]
if this limit exists. We fix $f>0$ and a sequence $\delta_N=\frac{1}{N}$.
Notice that for any $C>0$ the sequence $f_N^{(C)}:=f-C\delta_N$ satisfies
assumption of Proposition~\ref{p:incr} (see Remark \ref{rem14}). Therefore,
$\LL((f_N^{(C)}))$ is well defined.
Finally, let
\[
\II_f:= \liminf_{N\to\infty} \frac{\ln \P\{M_{N} \le f \}}{N}
; \qquad
\MM_f:= \limsup_{N\to\infty} \frac{\ln \P\{M_{N} \le f \}}{N}\, .
\]
We shall prove that
\[
   \II_f=\MM_f=\LL((f_N^{(1)})).
\]
Since we obviously have
\[
   \LL((f_N^{(1)}))\le \II_f \le \MM_f,
\]
it remains to prove that $\MM_f\le \LL((f_N^{(1)}))$.

By the well known log-concavity
of Gaussian measures (\cite{borell}), for each $N\in\N$ the function
$r\mapsto \ln \P(M_N\le r)$ is concave. In particular, for every $C>1$
we have
\begin{eqnarray*}
   &&\ln \P\{M_{N} \le f \} - \ln \P\{M_{N} \le f -\delta_N \}
\\
   &\le& \frac{1}{C} \left(  \ln \P\{M_{N} \le f \}
         - \ln \P\{M_{N} \le f - C\delta_N \}\right).
\end{eqnarray*}
Dividing by $N$ and taking $\limsup$ over $N$ yields
\[
   \MM_f - \LL((f_N^{(1)}))
   \le \frac{1}{C} \left(  \MM_f - \LL((f_N^{(C)})) \right)
   \le \frac{1}{C} \left(  \MM_f - \LL((f/2-\delta_N)) \right) .
\]
Recall that $0\ge \MM_f\ge \LL((f/2-\delta_N))>-\infty$ by Proposition \ref{p:incr}.
Therefore, letting $C\to\infty$ yields  $\MM_f - \LL((f_N^{(1)}))\le 0$, as required.
\end{proof}

\begin{proof}[ of Theorem~\ref{thm:fconst}]
Let the function $\functionc :(0,\infty)\mapsto (-\infty,0]$ be defined,
according to \eqref{flim}, as
\[
    \functionc(f):= \lim_{N\to\infty} \frac{\ln \P\{M_{N} \le f \}}{N}\ .
\]
As a pointwise limit of concave functions, $\functionc(\cdot)$ is itself
concave, hence it is continuous.

Now take a sequence $(f_N)$ satisfying our Theorem's asumption and denote
$f:=\lim_{N\to\infty} f_N\in (0,\infty)$. It is obvious that for any $\delta\in(0,f)$
\[
 \functionc(f-\delta) \le  \liminf_{N\to\infty} \frac{\ln \P\{M_{N} \le f_N \}}{N}
 \le  \limsup_{N\to\infty} \frac{\ln \P\{M_{N} \le f_N \}}{N}
 \le \functionc(f+\delta).
\]
By letting $\delta\to 0$ and using the continuity of $\functionc(.)$ we obtain
\[
 \functionc(f) \le  \liminf_{N\to\infty} \frac{\ln \P\{M_{N} \le f_N \}}{N}
 \le  \limsup_{N\to\infty} \frac{\ln \P\{M_{N} \le f_N \}}{N}
 \le \functionc(f).
\]
It follows that
\[
   \lim_{N\to\infty} \frac{\ln \P\{M_{N} \le f_N \}}{N} = \functionc(f).
\]

It remains to confirm that this limit is strictly negative assuming that
Kolmogorov criterion holds. In this case
$\sigma^2:=\Var(\xi_1|\xi_0,\xi_{-1},\xi_{-2},\ldots)>0$
(see \cite{brockwelldavis}). We obtain with Anderson's inequality (cf.\ (\ref{eqn:anderson})):
\begin{eqnarray*}
    && \P\{ \max_{1\le  n\le  N} |S_n|\le  f\}
\\
    &=& \E[ \P\{ \max_{1\le  n\le  N} |S_n|\le  f | \xi_{N-1},\xi_{N-2},\ldots\} ]
\\
    &=& \E[ \ind_{\max_{1\le  n\le  N-1} |S_n|\le  f} \,
    \P\{ |S_{N-1}+\xi_N|\le  f | \xi_{N-1},\xi_{N-2},\ldots\} ]
\\
    &\le& \E[ \ind_{\max_{1\le  n\le  N-1} |S_n|\le  f} \,\P\{ \sigma|\NN|\le  f \} ]
\\
     &=& \P \{ \max_{1\le  n\le  N-1} |S_n|\le  f \} \, \P\{ \sigma|\NN|\le  f \}
\\
&\le&\ldots
\\
&\le & \P\{ \sigma|\NN|\le f \}^N,
\end{eqnarray*}
where $\NN$ is a standard normal random variable. This shows that
\[
   \lim_{N\to\infty} \frac{1}{N}\ln \P\{ \max_{1\le n\le N} | S_n|\le  f\}
   \le  \ln \P\{\sigma|\NN|\le  f \} < 0.
\]
\end{proof}


\begin{rem} {\rm
There are various interesting open questions related to the constant 
$\functionc(f)$ from Theorem~\ref{thm:fconst}. If we consider it as a 
function $\functionc : (0,\infty) \to (-\infty,0]$, we may ask: 
1) Is it true that $\lim_{f\to 0} \functionc(f) = -\infty$? 
2) Is it true that $\lim_{f\to \infty} \functionc(f) = 0$?
3) Is it true that for every singular process $\functionc(f)=0$ for all $f>0$?
}
\end{rem}
\medskip

{\bf Acknowledgment.} \
M.\,Lifshits was supported by RFBR grant 16-01-00258.


\end{document}